\documentclass[11pt,twoside]{amsart}

\usepackage[
  paper=a4paper,
  headsep=15pt,text={155mm,230mm},centering
]{geometry}

\usepackage[bookmarks]{hyperref}
\usepackage{amssymb,amsxtra}
\usepackage{graphicx,xcolor}
\usepackage{float,array,calc,booktabs,stackrel,url}
\usepackage{enumitem}\setlist[enumerate,1]{font=\upshape}
\usepackage{mathtools}
\usepackage{tikz}
\usetikzlibrary{
  cd,
  calc,
  positioning,
  arrows,
  decorations.pathreplacing,
  decorations.markings,
}
\usepackage[mode=buildnew]{standalone}

\iftrue
    
    \makeatletter
    \def\@settitle{%
      \vspace*{-10pt}
      \begin{flushleft}%
        \LARGE\bfseries
        \strut\@title\strut
      \end{flushleft}%
    }
    \def\@setauthors{%
      \begingroup
      \def\thanks{\protect\thanks@warning}%
      \trivlist
      \raggedright
      \large \@topsep27\p@\relax
      \advance\@topsep by -\baselineskip
    \item\relax
      \author@andify\authors
      \def\\{\protect\linebreak}%
      \authors
      \ifx\@empty\contribs
      \else
      ,\penalty-3 \space \@setcontribs
      \@closetoccontribs
      \fi
      \normalfont
      \endtrivlist
      \endgroup
    }
    \def\@setaddresses{\par
      \nobreak \begingroup
      \small\raggedright
      \def\author##1{\nobreak\addvspace\smallskipamount}%
      \def\\{\unskip, \ignorespaces}%
      \interlinepenalty\@M
      \def\address##1##2{\begingroup
        \par\addvspace\bigskipamount\noindent
        \@ifnotempty{##1}{(\ignorespaces##1\unskip) }%
        {\ignorespaces##2}\par\endgroup}%
      \def\curraddr##1##2{\begingroup
        \@ifnotempty{##2}{\nobreak\noindent\curraddrname
          \@ifnotempty{##1}{, \ignorespaces##1\unskip}\/:\space
          ##2\par}\endgroup}%
      \def\email##1##2{\begingroup
        \@ifnotempty{##2}{\nobreak\noindent E-mail address%
          \@ifnotempty{##1}{, \ignorespaces##1\unskip}\/:\space
          \ttfamily##2\par}\endgroup}%
      \def\urladdr##1##2{\begingroup
        \def~{\char`\~}%
        \@ifnotempty{##2}{\nobreak\noindent\urladdrname
          \@ifnotempty{##1}{, \ignorespaces##1\unskip}\/:\space
          \ttfamily##2\par}\endgroup}%
      \addresses
      \endgroup
      \global\let\addresses=\@empty
    }
    \def\@setabstracta{%
      \ifvoid\abstractbox
      \else
      \skip@17pt \advance\skip@-\lastskip
      \advance\skip@-\baselineskip \vskip\skip@
      \box\abstractbox
      \prevdepth\z@ 
      \vskip-28pt
      \fi
    }
    \renewenvironment{abstract}{%
      \ifx\maketitle\relax
      \ClassWarning{\@classname}{Abstract should precede
        \protect\maketitle\space in AMS document classes; reported}%
      \fi
      \global\setbox\abstractbox=\vtop \bgroup
      \normalfont\small
      \list{}{\labelwidth\z@
        \leftmargin0pc \rightmargin\leftmargin
        \listparindent\normalparindent \itemindent\z@
        \parsep\z@ \@plus\p@
        
      }%
    \item[\hskip\labelsep\bfseries\abstractname.]%
    }{%
      \endlist\egroup
      \ifx\@setabstract\relax \@setabstracta \fi
    }

    \def\ps@headings{\ps@empty
      \def\@evenhead{%
        \setTrue{runhead}%
        \normalfont\scriptsize
        \rlap{\thepage}\hfill
        \def\thanks{\protect\thanks@warning}%
        \leftmark{}{}}%
      \def\@oddhead{%
        \setTrue{runhead}%
        \normalfont\scriptsize
        \def\thanks{\protect\thanks@warning}%
        \rightmark{}{}\hfill \llap{\thepage}}%
      \let\@mkboth\markboth
    }\ps@headings

    \def\section{\@startsection{section}{1}%
      \z@{-1.4\linespacing\@plus-.5\linespacing}{.8\linespacing}%
      {\normalfont\bfseries\Large}}
    \def\subsection{\@startsection{subsection}{2}%
      \z@{-.8\linespacing\@plus-.3\linespacing}{.5\linespacing\@plus.2\linespacing}%
      {\normalfont\bfseries\large}}
    \def\subsubsection{\@startsection{subsubsection}{3}%
      \z@{.7\linespacing\@plus.2\linespacing}{-1.5ex}%
      {\normalfont\bfseries}}
    \def\paragraph{\@startsection{paragraph}{4}%
      \z@{.7\linespacing\@plus.2\linespacing}{-1.5ex}%
      {\normalfont\itshape}}
    \def\@secnumfont{\bfseries}

    \renewcommand\contentsnamefont{\bfseries}
    \def\@starttoc#1#2{\begingroup
      \setTrue{#1}%
      \par\removelastskip\vskip\z@skip
      \@startsection{}\@M\z@{\linespacing\@plus\linespacing}%
      {.5\linespacing}{
        \contentsnamefont}{#2}%
      \ifx\contentsname#2%
      \else \addcontentsline{toc}{section}{#2}\fi
      \makeatletter
      \@input{\jobname.#1}%
      \if@filesw
      \@xp\newwrite\csname tf@#1\endcsname
      \immediate\@xp\openout\csname tf@#1\endcsname \jobname.#1\relax
      \fi
      \global\@nobreakfalse \endgroup
      \addvspace{32\p@\@plus14\p@}%
      \let\tableofcontents\rela\x
    }
    \def\contentsname{Contents}
    \def\l@section{\@tocline{2}{.5ex}{0mm}{5pc}{}}
    \def\l@subsection{\@tocline{2}{0pt}{2em}{5pc}{}}
    \makeatother

\fi

\def\to{\mathchoice{\longrightarrow}{\rightarrow}{\rightarrow}{\rightarrow}}
\makeatletter
\newcommand{\shortxra}[2][]{\ext@arrow 0359\rightarrowfill@{#1}{#2}}
\def\longrightarrowfill@{\arrowfill@\relbar\relbar\longrightarrow}
\newcommand{\longxra}[2][]{\ext@arrow 0359\longrightarrowfill@{#1}{#2}}
\renewcommand{\xrightarrow}[2][]{\mathchoice{\longxra[#1]{#2}}%
  {\shortxra[#1]{#2}}{\shortxra[#1]{#2}}{\shortxra[#1]{#2}}}
\makeatother

\makeatletter
\def\addtagsub#1{\let\oldtf=\tagform@\def\tagform@##1{\oldtf{##1}\hbox{$_{#1}$}}}
\makeatother

\makeatletter
\def\Nopagebreak{\@nobreaktrue\nopagebreak}
\makeatother

\makeatletter
\newtheoremstyle{theorem-giventitle}
        {}{}              
        {\itshape}                      
        {}                              
        {\bfseries}                     
        {.}                             
        {\thm@headsep}                             
        {\thmnote{\bfseries#3}}
\newtheoremstyle{theorem-givenlabel}
        {}{}              
        {\itshape}                      
        {}                              
        {\bfseries}                     
        {.}                             
        {\thm@headsep}                             
        {\thmname{#1}~\thmnumber{#3}\setcurrentlabel{#3}}
\newtheoremstyle{definition-giventitle}
        {}{}              
        {}                      
        {}                              
        {\bfseries}                     
        {.}                             
        {\thm@headsep}                             
        {\thmnote{\bfseries#3}}
\def\setcurrentlabel#1{\gdef\@currentlabel{#1}}
\makeatother

\newtheorem{theorem}{Theorem}[section]
\newtheorem{theoremalpha}{Theorem}

\newtheorem{proposition}[theorem]{Proposition}
\newtheorem{corollary}[theorem]{Corollary}
\newtheorem{lemma}[theorem]{Lemma}
\newtheorem{conjecture}{Conjecture}

\newtheorem{question}[conjecture]{Question}

\theoremstyle{definition}
\newtheorem{definition}[theorem]{Definition}

\newtheorem{remark}[theorem]{Remark}

\theoremstyle{theorem-giventitle}
\newtheorem{theorem-named}{}
\theoremstyle{theorem-givenlabel}
\newtheorem{theorem-labeled}{Theorem}

\theoremstyle{definition-giventitle}
\newtheorem{definition-named}{}
\newtheorem{conjecture-named}{}
\newtheorem{case-named}{}

\numberwithin{equation}{section}

\def\Z{\mathbb{Z}}
\def\R{\mathbb{R}}
\def\Q{\mathbb{Q}}
\def\C{\mathbb{C}}

\def\cP{\mathcal{P}}

\def\cN{\mathcal{N}}

\def\tilde{\widetilde}

\def\sm{\smallsetminus}
\DeclareMathOperator\Ker{Ker}
\DeclareMathOperator\Coker{Coker}
\def\Im{\operatorname{Im}}

\DeclareMathOperator\sign{sign}

\def\rank{\operatorname{rank}}
\def\lk{\operatorname{lk}}

\def\cupover#1{\mathbin{\mathop{\cup}_{#1}}}

\def\Bl{\mathop{B\ell}}
\def\rhot{\rho^{(2)}}

\def\cC{\mathcal{C}}

\def\cF{\mathcal{F}}
\def\cK{\mathcal{K}}

\def\cB{\mathcal{B}}

\def\C{\mathbb{C}}

\def\lt{L^2}

\makeatletter

\begin{document}

\title
  [Iterated satellite operators]
  {Iterated satellite operators on the knot concordance group}

\author{Jae Choon Cha}
\address{
  Center for Research in Topology and Department of Mathematics\\
  POSTECH\\
  Pohang 37673\\
  Republic of Korea
}
\email{jccha@postech.ac.kr}

\author{Taehee Kim}
\address{Department of Mathematics\\
  Konkuk University\\
  Seoul 05029\\
  Republic of Korea
}
\email{tkim@konkuk.ac.kr}

\thanks{\noindent
  The first named author was supported by the National Research Foundation of Korea grant RS-2019-NR039996.
  The second named author was supported by the National Research Foundation of Korea grant RS-2023-NR076429.
}

\def\subjclassname{\textup{2010} Mathematics Subject Classification}
\expandafter\let\csname subjclassname@1991\endcsname=\subjclassname
\expandafter\let\csname subjclassname@2000\endcsname=\subjclassname

\begin{abstract}
  We show that for a winding number zero satellite operator $P$ on the knot concordance group, if the axis of $P$ has nontrivial self-pairing under the Blanchfield form of the pattern, then the image of the iteration $P^n$ generates an infinite rank subgroup for each~$n$.
  Furthermore, the graded quotients of the filtration of the knot concordance group associated with $P$ have infinite rank at all levels.
  This gives an affirmative answer to a question of Hedden and Pinz\'{o}n-Caicedo in many cases.
  We also show that under the same hypotheses, $P^n$ is not a homomorphism on the knot concordance group for each~$n$.
  We use amenable $\lt$-signatures to prove these results.
\end{abstract}

\maketitle

\section{Introduction}
\label{section:introduction} 

Manifolds of dimensions 3 and 4 are key subjects in low dimensional topology.
For 3-manifolds, decomposition along surfaces is fundamental.
The prime decomposition theorem of Kneser and Milnor gives a unique decomposition, along spheres, into prime 3-manifolds.
For irreducible 3-manifolds, \emph{toral} decompositions are essential in the theory of Jaco-Shalen and Johannson and in Thurston's geometrization.
Knot theory has been a rich source of the toral decomposition:
tying a knot in a solid torus (called a \emph{pattern}) along another knot in $S^3$ (called a \emph{companion} or \emph{orbit}), one obtains a \emph{satellite knot}, whose exterior is decomposed, along a torus, into the exteriors of the pattern and companion.
Iteration of satellite constructions produces more sophisticated examples.

The 3-dimensional satellite construction is also useful in studying 4-dimensional topology, via knot concordance.
Two knots $K$ and $J$ in $S^3$ are \emph{concordant} if there is a properly embedded locally flat annulus in the 4-manifold $S^3\times [0,1]$ bounded by $-K\times 0$ and $J\times 1$.
(Knots and manifolds are oriented in this paper.)
The concordance classes form an abelian group under connected sum, which is called the \emph{knot concordance group}.
We denote it by~$\cC$.
A knot is \emph{slice} if it is concordant to the unknot.
The smooth version of $\cC$ is defined by requiring the annulus smoothly embedded.
Knot concordance has been studied intensively, particularly regarding its connection with fundamental questions in dimension~4.
But our understanding is largely limited, for both topological and smooth cases.

Since the 80s, satellite construction has been a major tool in understanding the difference between $\cC$ and the high dimensional knot concordance groups using the Casson--Gordon invariant.
See, e.g.,~\cite{Litherland:1984-1,Gilmer:1982-1,Gilmer:1983-1,Gilmer-Livingston:1992-1,livingston:1983-1}.
A breakthrough of Cochran, Orr and Teichner~\cite{Cochran-Orr-Teichner:2003-1} led us to use \emph{iterated} satellite constructions to understand ``higher order'' structures in~$\cC$ which are even invisible to the Casson--Gordon invariant.
Cochran, Harvey and Leidy investigated this direction successfully in~\cite{Cochran-Harvey-Leidy:2009-01,Cochran-Harvey-Leidy:2009-02}.
Briefly, iteration sends a companion knot deeper (from the boundary) in the toral decomposition in dimension 3, and it interplays elegantly with the filtrations of Cochran, Orr and Teichner which are geometrically defined via Whitney towers and gropes in 4-manifolds.

Also, satellite constructions exhibit the difference between topological and smooth categories in dimension~4.
A well-known example is the Whitehead doubling operator, which results in satellite knots with Whitehead pattern.
All Whitehead doubles are topologically slice by Freedman~\cite{Freedman:1984-1}, while Casson and Akbulut observed that many of these are not smoothly slice using Donaldson's work~\cite{Donaldson:1983-1}.
Since then, many authors used various satellite constructions, together with modern smooth invariants such as gauge theory, Heegaard Floer homology and Khovanov homology, to study the smooth concordance group of topologically slice knots, which measures the difference between the two categories.
See, e.g.,~\cite{Hedden:2007-1,Hedden-Kirk:2010-1,Hedden-Livingston-Ruberman:2010-01,Hedden-Kim-Livingston:2016-1}.
More recently, in~\cite{Cochran-Harvey-Horn:2012-1,Cha-Kim:2021-1,Cha:2021-1}, it was shown that \emph{iterated} satellite constructions do generate rich structures in deeper parts in the smooth concordance group of topologically slice knots, which most, if not all, known modern smooth invariants fail to detect.

\subsubsection*{Some conjectures and questions}

The growing successful applications motivate demands for a better understanding of the 4-dimensional nature of satellite construction.
In the introduction of~\cite{Hedden-Pinzon-Caicedo:2021-1}, Hedden and Pinz\'{o}n-Caicedo present an excellent discussion about this and formulate several intriguing conjectures.
For brevity, denote by $P(K)$ the satellite knot with pattern $P\subset S^1\times D^2$ and companion~$K\subset S^3$.
This means that $P(K)\subset S^3$ is the image of $P\subset S^1\times D^2$ under a diffeomorphism of $S^1\times D^2$ onto a tubular neighborhood of $K$ along the zero framing.
See Figure~\ref{figure:satellite-example}.
The association $K\mapsto P(K)$ descends to an operator $P\colon \cC\to \cC$ on the knot concordance group.
Let $P^n=P\circ\cdots \circ P\circ P$, the iterated satellite operator that applies $P$ $n$ times.

\begin{figure}[h]
  \includestandalone[scale=1]{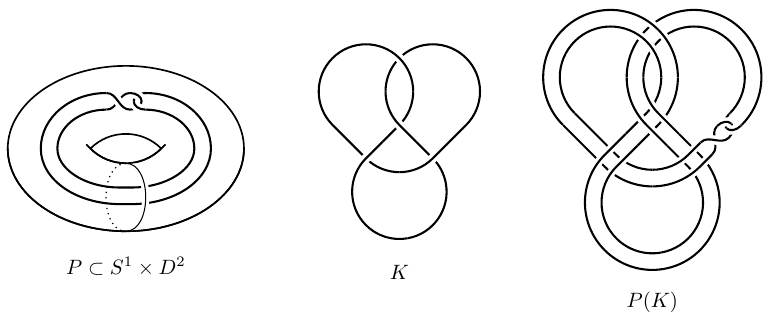}
  \caption{A satellite knot $P(K)$ with pattern $P$ and companion $K$.}
  \label{figure:satellite-example}
\end{figure}

In the literature, several authors constructed infinitely many knots of the form $P(K)$ which are linearly independent in~$\cC$, for a fixed~$P$.
For such $P$, the image of $P\colon \cC\to \cC$ generates an infinite rank subgroup.
Interestingly, for all known cases, the image of a satellite operator $P\colon \cC\to \cC$ is either large like this, or a constant.
This leads us to the following dichotomy conjecture.

\begin{conjecture}[\cite{Hedden-Pinzon-Caicedo:2021-1}]
  \label{conjecture:infinite-rank}
  Every satellite operator $P\colon \cC\to \cC$ is constant or has image generating an infinite rank subgroup in~$\cC$.
\end{conjecture}

For $P\subset S^1\times D^2$, the \emph{winding number} is defined to be the integer representing the class $[P] \in H_1(S^1\times D^2)=\Z$.
It was observed in \cite[Proposition~8]{Hedden-Pinzon-Caicedo:2021-1} that $P(\cC)$ always generates an infinite rank subgroup if the winding number is nonzero.
So Conjecture~\ref{conjecture:infinite-rank} is reduced to the winding number zero case. Recently, several authors produced partial answers for Conjecture~\ref{conjecture:infinite-rank}.
See, e.g.,\ \cite{Daemi-Imori-Sato-Scaduto-Taniguchi:2022-1,Dai-Hedden-Mallick-Stoffregen:2022-1,Livingston:2022-1}.

To generalize Conjecture~\ref{conjecture:infinite-rank} for iterated satellite operators, Hedden and Pinz\'{o}n-Caicedo consider a filtration associated with a satellite operator $P$,
\begin{equation}
  \cdots \subset \langle P^n(\cC) \rangle \subset \cdots \subset \langle P^2(\cC) \rangle \subset \langle P(\cC) \rangle \subset \langle P^0(\cC) \rangle = \cC,
  \label{equation:P-filtration}
\end{equation}
which they call the \emph{$P$-filtration}.
Here $P^0$ is the identity map, hence $P^0(K)=K$, and $\langle S\rangle$ is the subgroup generated by a subset $S\subset \cC$.

\begin{question}[\cite{Hedden-Pinzon-Caicedo:2021-1}]
  \label{question:rank-iterated-satellite}
  If a winding number zero satellite operator $P\colon \cC\to \cC$ is non-constant, does each associated graded quotient $\langle P^n(\cC)\rangle/\langle P^{n+1}(\cC)\rangle$, $n\ge 0$, have infinite rank?
\end{question}

For the nonzero winding number case, see our discussion given below, especially Proposition~\ref{proposition:nonzero-winding-number} in this section.

Note that an affirmative answer to Question~\ref{question:rank-iterated-satellite} implies that Conjecture~\ref{conjecture:infinite-rank} is true.

\subsubsection*{Blanchfield form and iterated satellite operators}

Our first main result presents an affirmative answer to Conjecture~\ref{conjecture:infinite-rank} and Question~\ref{question:rank-iterated-satellite} for a large family of satellite operators~$P$.
To state it, we use the following notations.
Let $U\subset S^3$ be the trivial knot.
Applying $P\subset S^1\times D^2$ to $U$ by identifying $S^1\times D^2$ with a tubular neighborhood $N$ of~$U$, one obtains the knot~$P(U)\subset S^3$.
The core $\eta$ of the solid torus $S^3\sm N$ is called the \emph{axis} of~$P$.
The winding number of $P$ is equal to $\lk(\eta, P(U))$.
So, when the winding number is zero, $\eta$~represents an element in the Alexander module $A = H_1(S^3 \sm P(U);\Z[t^{\pm 1}])$ of the knot~$P(U)$.
Let $\Bl\colon A\times A \to \Q(t)/\Z[t^{\pm 1}]$ be the Blanchfield form~\cite{Blanchfield:1957-1} of~$P(U)$.

\begin{theoremalpha}
  \label{theorem:main}
  If a satellite operator $P\colon \cC\to \cC$ has winding number zero and $\Bl(\eta,\eta)$ is nontrivial for the axis $\eta$, then $\langle P^n(\cC)\rangle/\langle P^{n+1}(\cC)\rangle$ has infinite rank for all $n\ge 0$. In particular, $\langle P(\cC)\rangle$ has infinite rank. 
\end{theoremalpha}

All results in this paper, including Theorem~\ref{theorem:main}, hold in the topological category, and consequently in the smooth category.

The Blanchfield form condition in Theorem~\ref{theorem:main} is satisfied in many cases.
See Remark~\ref{remark:nonvanishing-of-eta} and Lemma~\ref{lemma:nontrivial-self-blanchfield}.

We remark that some special cases of Theorem~\ref{theorem:main} were known earlier.
Cochran--Harvey--Leidy proved that $\langle P^n(\cC)\rangle/\langle P^{n+1}(\cC)\rangle$ has infinite rank for $n\ge 1$ and certain $P$ called \emph{robust doubling operators}.
A robust doubling operator is a winding number zero operator $P$ such that $P(U)$ is a ribbon knot satisfying certain conditions on the cyclicity and the behavior of a generator of the Alexander module and conditions on first order $L^2$-signatures~\cite[Definitions 2.6 and 7.2]{Cochran-Harvey-Leidy:2009-02}.
More generally, in~\cite{Cha:2010-01}, it was implicitly proved that Theorem~\ref{theorem:main} holds for winding number zero $P$ such that $\Bl(\eta,\eta)$ is nontrivial \emph{and $P(U)$ is slice}.
(See Section~\ref{subsection:slice-satellite-operators} for related discussions.)
Theorem~\ref{theorem:main} generalizes these.

Also, Hedden and Pinz\'{o}n-Caicedo proved a related result.
Recall that the branched double cover of a knot is a rational homology sphere, so the rational self-linking number of a framed simple closed curve in the cover is defined.
The main result of \cite{Hedden-Pinzon-Caicedo:2021-1} states that if $P$ is a satellite operator with winding number zero, and if the lift of the zero framed axis $\eta$ has nonzero rational self-linking in the double branched cover of $P(U)$, then $\langle P(\cC)\rangle$ is an infinite rank subgroup in the smooth knot concordance group.

This result in \cite{Hedden-Pinzon-Caicedo:2021-1} and Theorem~\ref{theorem:main} are complementary from multiple viewpoints.
Hedden--Pinz\'{o}n-Caicedo is restricted to the smooth category, while Theorem~\ref{theorem:main} applies to both smooth and topological categories. Hedden--Pinz\'{o}n-Caicedo applies, however, even to some operators $P$ such that $P(U)$ has trivial Alexander polynomial (in this case the image of $P$ contains only topologically slice knots by Freedman's work \cite{Freedman:1984-1}), while the Blanchfield form condition in Theorem~\ref{theorem:main} cannot be satisfied for such~$P$\@.
On the other hand, Hedden--Pinz\'{o}n-Caicedo does not apply to iterated satellite operators and, therefore, may not directly address the study of Question~\ref{question:rank-iterated-satellite}, because the rational self-linking number in the branched cover automatically vanishes for $P^n$ with $n\ge 2$ (see~\cite[Section~5.3]{Hedden-Pinzon-Caicedo:2021-1});
Theorem~\ref{theorem:main} detects the nontriviality of the iterated satellite $P$-filtration at each stage and shows that Conjecture~\ref{conjecture:infinite-rank} holds for~$P^n$ with $n$ arbitrary under the assumption $\Bl(\eta,\eta)\ne 0$.

The proof of Theorem~\ref{theorem:main} uses higher order $\lt$-signatures, and has some intriguing aspects compared to previously known techniques.
We discuss this in more detail at the beginning of Section~\ref{section:iterated-satellites-infinite-rank}.

\begin{remark}[A question on topologically slice iterated satellite knots]
  It appears to be challenging to understand the difference between the smooth and topological iterated satellite $P$-filtrations~\eqref{equation:P-filtration}.
  Specifically, we ask the following question.
  Suppose that $P$ is a winding number zero non-constant satellite operator on the smooth knot concordance group $\cC^{\mathrm{sm}}$ such that $P(\cC^{\mathrm{sm}})$ contains only topologically slice knots.
  (There are many such $P$, e.g.,\ a Whitehead doubling operator.)
  Does the smooth graded quotient $\langle P^n(\cC^{\mathrm{sm}})\rangle/\langle P^{n+1}(\cC^{\mathrm{sm}})\rangle$ have infinite rank for all~$n$?
  Due to the technical limitations of currently known smooth invariants, even detecting a single example of $P$ for which the answer is affirmative does not appear to be straightforward.
  A potentially useful approach may be combining higher order $\lt$-signatures with gauge theory or other modern smooth tools, as in~\cite{Cha-Kim:2021-1,Cha:2021-1} which use $\lt$-signatures and Heegaard Floer homology.
\end{remark}

For the zeroth graded quotient $\cC/\langle P(\cC)\rangle$, we observe the following fact.

\begin{proposition}
  \label{proposition:zeroth-quotient}
  If $P$ has winding number zero, $\cC / \langle P(\cC)\rangle$ has infinite rank.
\end{proposition}

Note that Proposition~\ref{proposition:zeroth-quotient} does not require $\Bl(\eta,\eta)\ne 0$ (cf.\ Theorem~\ref{theorem:main}).

In~\cite[p.~135]{Hedden-Pinzon-Caicedo:2021-1}, Hedden and Pinz\'{o}n-Caicedo observed that the smooth case of Proposition~\ref{proposition:zeroth-quotient} holds, using a result of Wang~\cite{Wang:2016-1}.
Proposition~\ref{proposition:zeroth-quotient} generalizes it to the topological category.

\subsubsection*{Non-zero winding number case}

It is well known that there are many satellite operators with winding number~$\pm 1$ for which the $P$-filtration~\eqref{equation:P-filtration} does not have infinite rank graded quotients in general.
That is, the answer to Question~\ref{question:rank-iterated-satellite} is not affirmative.
The simplest example is the connected sum operator $P_J(K)=J\# K$, where $J$ is a knot.
It satisfies $(P_J)^n(\cC)=\cC$ for all~$n\ge 0$.
It is not even known whether or not a winding number~$\pm 1$ satellite operator acts bijectively on the topological knot concordance group.

When the winding number is neither zero nor $\pm 1$, we observe that Question~\ref{question:rank-iterated-satellite} always has an affirmative answer, as stated below.

\begin{proposition}
  \label{proposition:nonzero-winding-number}
    If $P$ has winding number $\ne 0$, $\pm1$, $\langle P^n(\cC)\rangle / \langle P^{n+1}(\cC)\rangle$ has infinite rank for all $n\ge 0$, in both smooth and topological categories.
\end{proposition}

From this, it follows that the only remaining interesting case of Question~\ref{question:rank-iterated-satellite} is when the winding number is zero, which Theorem~\ref{theorem:main} addresses.

The authors could not find Propositions~\ref{proposition:zeroth-quotient} and~\ref{proposition:nonzero-winding-number} in the literature, so we give proofs in Section~\ref{section:non-zero-winding-number}.
The proofs use the classical knot signature function.

\subsubsection*{Iterated satellite operators which are not homomorphisms}

We also use the ideas and techniques of the proof of Theorem~\ref{theorem:main} to obtain more results on iterated satellite operators.
Our next result concerns a well-known conjecture of Hedden, which can be found in \cite[Conjecture~1]{Hedden-Pinzon-Caicedo:2021-1}.

\begin{conjecture}[Hedden]
  \label{conjecture:homomorphism}
  A satellite operator $P\colon \cC\to \cC$ is a homomorphism if and only if it is either zero, identity, or the involution induced by knot orientation reversal.
\end{conjecture}

Several authors verified the smooth version of Conjecture~\ref{conjecture:homomorphism} for special cases: for satellite operators obtained from a certain family of 2-bridge links by Chen~\cite{Chen:2019-2}, and for satellite operators satisfying certain conditions on linking numbers in branched covers by Lidman, Miller and Pinz\'{o}n-Caicedo~\cite{Lidman-Miller-Pinzon-Caicedo:2022-1} (also see \cite{Cahn-Kjuchukova:2023-1,Johanningsmeier-Kim-Miller:2023-1}).
In the topological (and consequently smooth) category, Miller verified Conjecture~\ref{conjecture:homomorphism} when the first homology of the $p$-fold branched cover of $P(U)$ is nontrivial and generated by the lifts of the axis $\eta$ for some prime $p$ dividing the winding number~\cite[Theorem~A]{Miller:2019-01}\@.
Miller also verified Conjecture~\ref{conjecture:homomorphism} for infinitely many satellite operators $P$ such that $P(U)$ is slice and $P(\cC)$ lies arbitrarily deep in the Cochran--Orr--Teichner filtration~\cite[Theorem~B]{Miller:2019-01}\@.

We give the following new evidence for Conjecture~\ref{conjecture:homomorphism}.

\begin{theoremalpha}
  \label{theorem:non-homomorphism}
  Let $P\colon \cC \to \cC$ be a winding number zero satellite operator with axis~$\eta$. If $\Bl(\eta, \eta)\ne 0$, then $P^n\colon \cC\to \cC$ is not a homomorphism for all $n\ge 1$.
\end{theoremalpha}

Theorem~\ref{theorem:non-homomorphism} generalizes the winding number zero case of~\cite[Theorem~A]{Miller:2019-01}, since the hypothesis on the lifts of $\eta$ in the latter implies the hypothesis $\Bl(\eta,\eta)\ne 0$ in the former.
Also, Theorem~\ref{theorem:non-homomorphism} readily implies~\cite[Theorem~B]{Miller:2019-01}.
In fact, \cite[Theorem~B]{Miller:2019-01} is obtained from a more technical result given in~\cite[Corollary~4.5]{Miller:2019-01}, which our result implies.
For a more detailed discussion, see Remark~\ref{remark:non-homomorphism-generalization}~(2).

\subsubsection*{Distinction of satellite operators}

We also address the following question which Cochran, Davis and Ray asked and discussed in~\cite[p.~963]{Cochran-Davis-Ray:2014-1}.

\begin{question}[\cite{Cochran-Davis-Ray:2014-1}]
  \label{question:distinct-satellite-operators}
  When do two patterns $P$, $Q\subset S^1\times D^2$ give distinct satellite operators $P$, $Q\colon \cC \to \cC$?
\end{question}

Regarding the case of iterated satellite operators, Franklin's result in~\cite[Theorem~3.1]{Franklin:2013-1} implies that $P^2$ and $Q^2$ are distinct on $\cC$ for certain winding number zero $P$, $Q\subset S^1\times D^2$ with $P(U)$, $Q(U)$ ribbon.

In the result below, we consider higher iterated satellite operators.

\begin{theoremalpha}
  \label{theorem:coprime-satellites}

  Suppose that $P\subset S^1\times D^2$ has winding number zero and $\Bl(\eta, \eta)$ is nontrivial for the axis $\eta$ of~$P$.
  For any winding number zero $Q\subset S^1\times D^2$ such that the Alexander polynomials of $P(U)$ and $Q(U)$ are relatively prime in $\Q[t^{\pm 1}]$, the operators $P^n$, $Q\colon \cC \to \cC$ are distinct for all $n\ge 0$.
\end{theoremalpha}

Note that it is an immediate consequence of Theorem~\ref{theorem:main} that if $P$ has winding number zero and $\Bl(\eta,\eta)$ is nontrivial for the axis $\eta$, then the operators $P^n$ $(n\ge 1)$ are mutually distinct.
From this and Theorem~\ref{theorem:coprime-satellites}, we readily obtain the following.

\begin{theorem-named}[Corollary to Theorems~\ref{theorem:main} and~\ref{theorem:coprime-satellites}]
  If $P$, $Q\colon \cC \to \cC$ are winding number zero satellite operators with axes $\eta$, $\zeta$ such that $\Bl(\eta,\eta)\ne 0$, $\Bl(\zeta, \zeta)\ne 0$ and the Alexander polynomials of $P(U)$ and $Q(U)$ are relatively prime, then the operators in the collection $\{P^n\}_{n\ge 0} \cup \{Q^m\}_{m\ge 0}$ are mutually distinct.
\end{theorem-named}

\subsubsection*{An application}

Finally, we remark that the techniques discussed in this paper appear to be useful for other applications as well.
As an example, we investigate concordance classes of knots with the same Seifert form.
It is Freedman's well-known result that a knot with trivial Alexander polynomial is topologically slice~\cite{Freedman:1984-1}.
It leads us to ask whether or not there is any other Alexander polynomial, or Seifert form, that determines a topological knot concordance class.
The main results of Livingston~\cite{Livingston:2002-1} and T.~Kim~\cite{Kim:2005-1} answer the question in the negative:
for each Seifert form $V$ with nontrivial Alexander polynomial, there are infinitely many pairwise non-concordant knots $K_i$ that have Seifert form~$V$.
In Appendix~\ref{section:slice-satellite-operators}, we give a quick proof of the following generalization using our method for satellite operators.

\begin{theorem-named}[Corollary~\ref{corollary:infinite-rank-fixed-Seifert-form}]
  Suppose that $V$ is a Seifert form of a knot $K$ with nontrivial Alexander polynomial.
  Then, there exist infinitely many knots $K_i$ with the same Seifert form $V$ which are linearly independent in the knot concordance group~$\cC$. 
\end{theorem-named}

Our proof directly constructs the desired knots in the form $K_i=P(J_i)$ by applying a single fixed satellite operator~$P$, while earlier methods in~\cite{Livingston:2002-1,Kim:2005-1} depend on more sophisticated constructions.

\subsubsection*{Organization of the paper}
In Section~\ref{section:non-zero-winding-number} we prove Propositions~\ref{proposition:zeroth-quotient} and~\ref{proposition:nonzero-winding-number}.
In Section~\ref{section:preliminary}, we quickly review $\lt$-signatures and their applications that give slice obstructions.
We prove Theorems~\ref{theorem:main}, \ref{theorem:non-homomorphism}, and \ref{theorem:coprime-satellites} in Sections~\ref{section:iterated-satellites-infinite-rank}, \ref{section:non-homomorphism}, and \ref{section:coprime-Alexander-polynomial}, respectively.
In Appendix~\ref{section:slice-satellite-operators}, which is independent of the rest of the paper, we discuss alternative proofs of certain special cases of Theorem~\ref{theorem:main} using a simple modification of the proof of \cite[Theorem~4.11]{Cha:2010-01}, which motivated our work. 

In this paper, homology groups $H_*(-)$ are with integer coefficients unless mentioned otherwise, and a curve denotes its homotopy and homology classes as well.

\section{Graded quotients and classical signature invariants}
\label{section:non-zero-winding-number}

In this section, we give proofs of Propositions~\ref{proposition:zeroth-quotient} and~\ref{proposition:nonzero-winding-number}.
Other sections of this paper are independent of this section.
Readers may skip to Section~\ref{section:preliminary} if they want to read the proofs of other results (e.g.\ Theorem~\ref{theorem:main}) first.

Recall that the \emph{Levine-Tristram signature function} of a knot $K$ is defined by
\[
  \sigma_K(\omega) = \sign\big((1-\omega)A + (1-\overline\omega)A^T\big)
\]
for $\omega\in S^1\subset \C$, where $A$ is a Seifert matrix of~$K$.
We will use the \emph{signature jump function} $\delta_K\colon \R\to \Z$ that is defined by
\[
  \delta_K(\theta) = \lim_{s \to \theta^+} \sigma_K(e^{s\sqrt{-1}}) - \lim_{s\to \theta^-}\sigma_K(e^{s\sqrt{-1}}).
\]
It is well-known that $\delta_K$ has period~$2\pi$, $\delta_K(\theta) = \delta_K(-\theta)$ and $\delta_K(\theta)\ne 0$ only if $e^{\theta\sqrt{-1}}$ is a zero of the Alexander polynomial of~$K$.
In addition, $\delta_\bullet(\theta)$ is invariant under concordance and additive under connected sum for each fixed~$\theta$.
That is, $\delta_\bullet(\theta)$ induces an abelian group homomorphism $\cC\to \Z$.
We will also use the following fact.

\begin{lemma}[{\cite{Cha-Livingston:2002-1}}]
  \label{lemma:signature-jump-realization}
  For any given $\theta_0\in (0,\pi)$ and $\epsilon\in(0,\theta_0)$, there exist $\theta_1\in (\theta_0-\epsilon,\theta_0)$ and a knot $K$ such that $\delta_K(\theta_1)\ne 0$ and $\delta_K=0$ on $[0,\pi]\smallsetminus \{\theta_1\}$.
\end{lemma}

It follows from~\cite[Proof of Theorem~1]{Cha-Livingston:2002-1}.
See also the first sentence of~\cite[Proof of Lemma~5.6]{Cha:2009-1}.

\begin{proof}[Proof of Proposition~\ref{proposition:zeroth-quotient}]
  Let $P$ be a satellite operator with winding number zero.
  Our goal is to show that the quotient $\cC / \langle P(\cC)\rangle$ has infinite rank.

  For brevity, let $R=P(U)$.
    Since the set $\{\theta\in (0,\pi)\mid \delta_R(\theta)\ne 0\}$
  is finite, one obtains the following fact by repeatedly applying Lemma~\ref{lemma:signature-jump-realization}:
  there exist a sequence of knots $K_1,K_2, \dots$ and a strictly decreasing sequence of real numbers $\phi_1, \phi_2, \ldots$ in $(0,\pi)$ such that on $(0,\pi)$ the function $\delta_{K_i}(\theta)$ is nonzero at and only at $\theta=\phi_i$, and $\delta_R(\phi_i)=0$ for all~$i$. 
  
  We will show that the knots $K_i$ are linearly independent in $\cC/\langle P(\cC)\rangle$.
  Suppose that a finite linear combination $K =\#_i\, a_i K_i$ ($a_i\in \Z$) of the $K_i$ lies in $\langle P(\cC)\rangle$.
  Note that for any knot $J$, the knots $P(J)$ and $R=P(U)$ have the same Seifert matrix since $P$ has winding number zero, and consequently $\delta_{P(J)} = \delta_R$.
  Since $K\in \langle P(\cC)\rangle$, it follows that $\delta_K = a\delta_R$ for some integer~$a$.
  In particular $\delta_K(\phi_i)=a\delta_R(\phi_i)=0$ by the choice of $\phi_i$.
  Since $\delta_{K_j}(\phi_i)=0$ for $i\ne j$, we have $a_i \delta_{K_i}(\phi_i) = \delta_K(\phi_i)=0$.
  Since $\delta_{K_i}(\phi_i) \ne 0$, it follows that $a_i=0$ for all~$i$.
  This completes the proof.
  \end{proof}

\begin{proof}[Proof of Proposition~\ref{proposition:nonzero-winding-number}]
  Let $P$ be a satellite operator with winding number $d\ne0,\pm 1$.
  Proposition~\ref{proposition:nonzero-winding-number} asserts that $\langle P^n(\cC)\rangle / \langle P^{n+1}(\cC)\rangle$ has infinite rank.
  
  Let $R=P(U)$.
  Suppose $d\ge 2$.
  The case $d\le -2$ can be proved similarly.

  By \cite[Theorem~2]{Litherland:1979-1} (see also \cite[Theorem~2.1]{Cha-Ko:2000-1}), the signature jump function of a satellite knot $P(K)$ is given by
  \begin{equation}\label{equation:reparametrization}
    \delta_{P(K)}(\theta)=\delta_R(\theta) + \delta_K(d\theta).
  \end{equation}
  By applying this repeatedly, we have
  \[
    \begin{aligned}
      \delta_{P^n(K)}(\theta) &= \delta_R(\theta)+\delta_{P^{n-1}(K)}(d\theta) \\
      &= \cdots = \delta_R(\theta)+\delta_R(d\theta)+\cdots+\delta_R(d^{n-1}\theta)+\delta_K(d^n\theta).
    \end{aligned}
  \]
  For brevity, define $f_n\colon \R\to \Z$ ($n\ge 0$) by 
  \[
  f_n(\theta) = \delta_R(\theta)+\delta_R(d\theta)+\cdots+\delta_R(d^n\theta)
  \]
  so that $\delta_{P^{n+1}(K)}(\theta)=f_n(\theta) + \delta_K(d^{n+1}\theta)$ for $n\ge 0$.

  The following observation on knots in $\langle P^{n+1}(\cC)\rangle$ will be useful.
  If $K\in \langle P^{n+1}(\cC)\rangle$, then $K$ is concordant to $\#_i\, a_i K_i$ for some $a_i\in \Z$ and $K_i\in P^{n+1}(\cC)$.
  So
  \begin{equation}
    \delta_K(\theta)-a f_n(\theta)=\sum_i a_i\delta_{K_i}(d^{n+1}\theta).
    \label{equation:iterated-sign}    
  \end{equation}
  where $a=\sum_i a_i\in\Z$.
  Since $\delta_{K_i}$ have period $2\pi$, it follows that the function $\delta_K - af_n\colon \R\to \Z$ has period $2\pi/d^{n+1}$. 

  Note that the set
  \[
  S=\{\theta\in (0,\pi)\mid f_n(\theta)\ne 0 \text{ or } f_{n-1}(\theta)\ne 0\}
  \]
  is finite.
  By using this and Lemma~\ref{lemma:signature-jump-realization} repeatedly, choose knots $J_1, J_2,\dots$ and real numbers $\phi_1, \phi_2, \dots$ such that
  \[
    \min(2\pi/d, \pi/2) > \phi_1>\phi_2>\cdots >0,
  \]
  $\phi_i/d^n\notin S$, $\phi_i/d^n+ 2\pi/d^{n+1}\notin S$, and on $(0,\pi)$ the function $\delta_{J_i}$ is nonzero at and only at~$\phi_i$.
  For later use, note that $\delta_{J_i}=0$ on $(0,3\pi/2)$ except at~$\phi_i$, since $0<\phi_i<\pi/2$, $\delta_{J_i}(\pi+\theta)=\delta_{J_i}(\pi-\theta)$ and $\delta_{J_i}(\pi)=0$.

  We claim that the knots $P^n(J_i)$ are linearly independent in $\langle P^n(\cC)\rangle / \langle P^{n+1}(\cC)\rangle$.
  Let 
  \[
  J=\#_i\, a_i P^n(J_i)
  \] 
  where $a_i$ are integers not all of which are zero.
  By the above observation, it suffices to show that $\delta_J-af_n$ does not have period $2\pi/d^{n+1}$ for any integer $a$.

    By~\eqref{equation:iterated-sign} and by our choices of $S$ and $\phi_i$, if $a_j\ne 0$ for some $j$, then we have
  \[
  \left(\delta_J-af_n\right)(\phi_j/d^n) = \sum_ia_i\delta_{J_i}(\phi_j) = a_j\delta_{J_j}(\phi_j)\ne 0.
  \]
  On the other hand, by~\eqref{equation:iterated-sign},
  \[
  \left(\delta_J-af_n\right)(\phi_j/d^n+2\pi/d^{n+1}) = \sum_ia_i\delta_{J_i}(\phi_j+2\pi/d).
  \]
  Since $d\ge 2$, our choice of $\phi_i$ and $\phi_j$ ensures that $\phi_i < 2\pi/d < \phi_j + 2\pi/d < 3\pi / 2$, and thus $\phi_i\ne \phi_j+2\pi/d$ for all $i$.
  So, by our choice of $J_i$, we have $\delta_{J_i}(\phi_j+2\pi/d)=0$ for all~$i$.
  Therefore $\left(\delta_J-af_n\right)(\phi_j/d^n+2\pi/d^{n+1})=0$.
  It follows that $\delta_J-a f_n$ does not have period $2\pi/d^{n+1}$ for any integer~$a$.
\end{proof}

\section{\texorpdfstring
  {Preliminaries on $L^2$-signature obstructions}
  {Preliminaries on L2-signature obstructions}}
\label{section:preliminary}

In this section, we review $L^2$-signatures and their applications as a sliceness obstruction.

Briefly, the definition of the invariants is as follows.
Let $M$ be a closed 3-manifold and~$G$ a group.
For a homomorphism $\phi\colon \pi_1(M)\to G$, the Cheeger--Gromov $L^2$ $\rhot$-invariant of $(M,\phi)$ is defined as follows.
It is known that there exist a compact 4-manifold $W$ and a group $\Gamma$ into which $G$ embeds such that $\partial W = M$ and $\pi_1(M) \to G\hookrightarrow \Gamma$ extends to a homomorphism $\psi\colon \pi_1(W)\to \Gamma$.
Let $\cN \Gamma$ be the group von Neumann algebra of~$\Gamma$.
It is a ring with involution that contains~$\Z\Gamma$, so that the $\cN\Gamma$-valued intersection form on the homology $H_2(W;\cN \Gamma)$ is defined.
The \emph{$L^2$-signature} $\sign_\Gamma^{(2)}(W)$ of the intersection form is defined as a real number.
Let $\sign (W)$ be the ordinary signature of~$W$.
Then the \emph{Cheeger--Gromov $\rhot$-invariant of $M$ associated with $\phi$} is defined by
\[
  \rhot(M,\phi) = \sign_\Gamma^{(2)}(W) - \sign(W).
\]
This is well-defined, independent of the choice of $\Gamma$ and~$W$.
We omit details but summarize properties of $L^2$-signatures and $\rhot$-invariants which we will use in this paper below.
For more about these, see, e.g.,\ \cite{Chang-Weinberger:2003-1,Cochran-Orr-Teichner:2003-1,Lueck-Schick:2003-1,Cochran-Teichner:2003-1,Harvey:2006-1,Cha:2014-2}.

We will use later the following $L^2$-version of the Novikov additivity to compute $\rhot(M,\phi)$.
Let $W_1$ and $W_2$ be compact 4-manifolds and $W=W_1\cup_M W_2$ where $M$ is a common boundary component of $W_1$ and~$W_2$.
Let $\phi\colon \pi_1(W)\to \Gamma$ be a homomorphism and view $W_1$ and $W_2$ as manifolds over $\Gamma$ by composing $\phi$ with inclusion-induced homomorphisms.
Then we have
\[
\sign_\Gamma^{(2)}(W) = \sign_\Gamma^{(2)}(W_1)+\sign_\Gamma^{(2)}(W_2).
\]
We note that it reduces to the usual non-$L^2$ Novikov additivity when $\Gamma$ is a trivial group.

It is known that $\rhot$-invariants over certain amenable groups give obstructions to topological sliceness~\cite{Cha-Orr:2009-01,Cha:2010-01}.
We will use the following version, which uses solvable groups that may have torsion.
To state it, denote by $M(K)$ the zero-framed surgery manifold of a knot~$K$.
Let $G^{(n)}$ be the \emph{$n$th derived subgroup} of a group~$G$, i.e.,\ $G^{(0)}=G$ and $G^{(n+1)}=[G^{(n)},G^{(n)}]$.
Also, we need two terms: an \emph{integral solution} bounded by $M(K)$, and groups in \emph{Strebel's class}. Those will be discussed below the statement. 

\begin{theorem}[{Amenable Signature Theorem \cite[Theorem~1.3]{Cha:2010-01}}]\label{theorem:vanishing-rho-invariants}
Let $n\ge 0$ be an integer.
Let $G$ be a group lying in Strebel's class $D(R)$ where $R$ is $\Q$ or $\Z_p$, $p$ prime, and $G^{(n+1)}=\{e\}$.
Suppose that $K$ is a knot admitting an integral $n.5$-solution $W$ bounded by~$M(K)$.
If a homomorphism $\phi\colon \pi_1(M(K))\to G$ extends to $\pi_1(W)$ and sends a meridian of $K$ to an infinite order element, then $\rhot(M(K)),\phi)=0$.
\end{theorem}

We remark that $G$ is solvable in the above statement and thus $G$ is automatically amenable.

Briefly, for an integer $n$, a knot $K$ is \emph{integrally $n$-solvable} if $M(K)$ bounds an integral $n$-solution, where an \emph{integral $n$-solution} is a 4-manifold $W$ whose intersection form over $\Z[\pi_1W/\pi_1 W^{(n)}]$ admits a certain ``lagrangian'' with ``dual.''
The definition for $h=n.5$ is similar, involving both $\Z[\pi_1W/\pi_1 W^{(n)}]$ and $\Z[\pi_1W/\pi_1 W^{(n+1)}]$.
A slice knot is integrally $h$-solvable for all $h$ ($h=n$ or~$n.5$).
We omit details, which the reader can find in~\cite[Definition~3.1]{Cha:2010-01}.
Instead, we will use the following facts only, in addition to Theorem~\ref{theorem:vanishing-rho-invariants}.
Let $\cF_h\subset \cC$ be the set of concordance classes of integrally $h$-solvable knots.
It is known that $\cF_h$ is a subgroup of~$\cC$.

\begin{lemma}\label{lemma:construction-of-n-solvable-knots}
Let $P\colon \cC\to \cC$ be a satellite operator with axis~$\eta$. Let $n$ be a positive integer.
\begin{enumerate}
	\item If $P(U)\in \cF_n$ and $\eta \in \pi_1(S^3\sm P(U))^{(n)}$, then $\langle P(\cC)\rangle \subset \cF_n$.
	\item If $P$ has winding number zero and $P(U)$ is slice, then $\langle P^n(\cC)\rangle \subset \cF_n$ for all $n$. 
\end{enumerate}
\end{lemma}
\begin{proof}
Property~(1) follows from \cite[Proposition~3.1]{Cochran-Orr-Teichner:2004-1}; 
the proof of \cite[Proposition~3.1]{Cochran-Orr-Teichner:2004-1} still holds when ``$n$-solvable'' is replaced by ``integrally $n$-solvable.'' 
(Here, one uses that every knot $K$ is integrally 0-solvable: an integral 0-solution, which has signature zero, is obtained by taking a connected sum of $\pm \C P^2$'s with a 4-manifold $W$ bounded by $M(K)$ such that the inclusion induces $H_1(M(K))\cong H_1(W)$.)

The axis of $P^n$ lies in $\pi_1(S^3\sm P^n(U))^{(n)}$ (e.g., see \cite[Lemma~4.9]{Cha:2010-01}).
So, property~(2) follows from~(1). 
\end{proof}

\begin{remark}
  The notion of an $h$-solvable knot was first introduced by Cochran, Orr, and Teichner~\cite{Cochran-Orr-Teichner:2003-1}.
  The notion of \emph{integrally} $h$-solvable knots that we use here is a slightly modified version.
  An $h$-solvable knot is integrally $h$-solvable but the converse is not true in general.
  For instance, a knot is 0-solvable if and only if the Arf invariant is trivial \cite[Remark~8.2]{Cochran-Orr-Teichner:2003-1}, but every knot is integrally $0$-solvable;
  see the proof of Lemma~\ref{lemma:construction-of-n-solvable-knots} above.
  This difference might look technical but is useful in our arguments.
  We note that known $L^2$-signature obstructions to $h$-solvablity in \cite{Cochran-Orr-Teichner:2003-1} and subsequent literature are obstructions to integral $h$-solvablity.
\end{remark}

Strebel's class $D(R)$ is defined in \cite{Strebel:1974-1} (see also \cite{Cha-Orr:2009-01}). We need only the following lemma that gives a class of groups in Strebel's class to which Theorem~\ref{theorem:vanishing-rho-invariants} applies.

\begin{lemma}\label{lemma:amenable-D(R)-group} Let $G$ be a solvable group with a subnormal series
\[
\{e\}=G_n \subset G_{n-1}\subset \cdots \subset G_1\subset G_0=G.
\]
\begin{enumerate}
	\item If every graded quotient $G_i/G_{i+1}$ is torsion-free abelian (i.e.,\ $G$ is poly-torsion-free-abelian), then $G$ lies in $D(\Q)$~\cite[p.~305]{Strebel:1974-1}.
	\item Let $p$ be a prime.
  If every graded quotient $G_i/G_{i+1}$ is abelian and has no torsion coprime to $p$ (i.e.,\ elements in $G_i/G_{i+1}$ have order a power of $p$ or~$\infty$), then $G$ lies in $D(\Z_p)$~\cite[Lemma~6.8]{Cha-Orr:2009-01}.
\end{enumerate}
\end{lemma}

The following lemma states some known properties of the $\rhot$-invariant that we will use in this paper.
For proofs, see \cite[Lemma~5.4]{Cochran-Orr-Teichner:2003-1}, \cite[Section~2]{Cochran-Orr-Teichner:2004-1}, and \cite[Lemma~8.7]{Cha-Orr:2009-01}.
Recall that $\sigma_K$ denotes the Levine-Tristram signature function for a knot~$K$.

\begin{lemma}
  \label{lemma:properties-of-rho-invariants}
    \begin{enumerate}
    \item If $i\colon G\to \Gamma$ is a monomorphism, $\rhot(M, \phi)=\rhot(M,i\circ \phi)$.
    \item If $\phi$ is a trivial homomorphism, $\rhot(M,\phi)=0$.
    \item $\rhot(-M,\phi)=-\rhot(M,\phi)$.
    \item If $\phi\colon \pi_1(M(K))\to \Z$ is surjective, then 
    \[
    \rhot(M(K), \phi) = \int_{S^1}\sigma_K(\omega) \, d\omega
    \] 
    where $S^1$ is the unit circle in $\C$ and the integral is normalized so that $\int_{S^1} d\omega = 1$.
    \item If $\phi\colon \pi_1(M(K))\to \Z_p$ is surjective, then 
    \[
    \rhot(M(K),\phi)=\frac1p \sum_{i=0}^{p-1}\sigma_K(\omega^i)
    \] 
    where $\omega=\textup{exp}(2\pi\sqrt{-1}/p)\in \C$.
  \end{enumerate}
\end{lemma}

For a satellite operator $P$ and a knot $K$, we will need to estimate the $\rhot$-invariants of the 3-manifold~$M(P(U))$.
The following result of Cheeger and Gromov gives a bound of the $\rhot$-invariants of a fixed 3-manifold.

\begin{theorem}[\cite{Cheeger-Gromov:1985-1}]
  \label{theorem:universal-bound}
  For each closed 3-manifold $M$, there exists a constant $C_M$ such that $|\rhot(M,\phi)|\le C_M$ for all homomorphisms $\phi$ of $\pi_1(M)$ to any group.
\end{theorem}

Several constructions of knots in later sections depend on the constant~$C_M$.
One can make those explicit by using the following quantitative estimate:
when $M=M(K)$ is the zero-framed surgery manifold of a knot $K$ with crossing number $c(K)$, Theorem~\ref{theorem:universal-bound} holds for $C_{M} = 7\cdot 10^7 \cdot c(K)$~\cite{Cha:2014-2}.

\section{Iterated satellite operators of infinite rank}
\label{section:iterated-satellites-infinite-rank}

In this section, we prove Theorem~\ref{theorem:main}\@, which asserts the following:
let $P\colon \cC\to \cC$ be a winding number zero satellite operator with axis $\eta$.
Suppose that $\Bl(\eta,\eta)$ is nontrivial, where $\Bl$ is the Blanchfield form of~$P(U)$.
Then, for each $n\ge 1$, the quotient $\langle P^n(\cC)\rangle/\langle P^{n+1}(\cC)\rangle$ has infinite rank.

In the proof, we use $\lt$-signature invariants and invoke the amenable signature theorem (Theorem~\ref{theorem:vanishing-rho-invariants}).
Our method has a distinct aspect compared with previously known techniques.
In $\lt$-signature arguments in the literature, to prove the linear independence of satellite knots of the form $P^n(K)$ for fixed $P^n$, it was needed to assume that \emph{$P(U)$ is slice}; without this, it seemed infeasible to compute and control the involved $\lt$-signatures. 

In our case, $P(U)$ is not slice in general.
Because of the above issue, we are unable to directly present an infinite family of knots of the form $P^n(K)$ which is linearly independent in $\langle P^n(\cC)\rangle/\langle P^{n+1}(\cC)\rangle$ (or even in $\cC$).
Instead, in this section, we develop an improved $\lt$-signature technique that works for $P$ with $P(U)$ non-slice, to prove the following:
for each $m\ge 1$, there are knots $K_1$,~\dots, $K_m$ such that $P^n(K_1)$,~\dots, $P^n(K_m)$ are linearly independent in $\langle P^n(\cC)\rangle/\langle P^{n+1}(\cC)\rangle$.

We remark that \emph{the knots $K_i$ depend on the value of~$m$.}
Allowing this dependency is crucial in our arguments, as seen in the proof of Proposition~\ref{proposition:infinite-rank-iteration} below.
(The step where the choice of $K_i$ depends on $m$ is Lemma~\ref{lemma:knots-J_0^i}.)
Despite the dependency on~$m$, it follows that $\langle P^n(\cC)\rangle/\langle P^{n+1}(\cC)\rangle$ has a subgroup of rank $m$ for each $m\ge 1$, so that one concludes that the rank is infinite.

We note that the following problem naturally arises, regarding an explicit generator set:
for any winding number zero non-constant satellite operator $P$, exhibit a sequence of explicit knots $K_1, K_2, \ldots$ such that $P^n(K_1), P^n(K_2), \ldots$ are linearly independent in $\langle P^n(\cC)\rangle/\langle P^{n+1}(\cC)\rangle$.

\begin{remark} \label{remark:nonvanishing-of-eta}
  The nontriviality condition on the Blanchfield pairing in Theorem~\ref{theorem:main} is satisfied in many cases. Since $\Bl$ is nonsingular, $\Bl(\eta,\eta)\ne 0$ if $\eta$ generates $H_1(S^3 \sm~P(U);\Q[t^{\pm 1}])$. 
  For a prime power~$q$, let $\Sigma_q(P(U))$ denote the $q$-fold cyclic cover of $S^3$ branched over $P(U)$ and let $\eta_i \in H_1(\Sigma_q(P(U)))$, $1\le i\le q$, be the homology classes of the $q$ lifts of $\eta$ to $\Sigma_q(P(U))$.
  Let $\lk$ be the $\Q/\Z$-valued linking form on $H_1(\Sigma_q(P(U)))$.
  It is well known that for some prime power $q$, if $\lk(\eta_i,\eta_j)$ is nontrivial for some $i$ and $j$, then $\Bl(\eta,\eta)$ is nontrivial (e.g.\ see \cite[Subsection~2.6]{Friedl:2003-7}).
  In particular, if the group $H_1(\Sigma_q(P(U)))$ is nontrivial and generated by $\eta_1,\ldots,\eta_q$, then $\Bl(\eta,\eta)$ is nontrivial. We also note the following: one way of getting a satellite operator from a given knot $\tilde{P}$ in $S^3$ is to choose a curve $\eta$ in $S^3\sm \tilde{P}$ which is unknotted in $S^3$, and take the exterior of $\eta$ in $S^3$. Now, if $H_1(S^3 \sm \tilde{P};\Q[t^{\pm 1}])$ is nontrivial, then one can see that there always exists a curve $\eta$ with $\lk(\eta,\tilde{P})=0$ in $S^3$ such that $\Bl(\eta,\eta)$ is nontrivial, since $\Bl$ is nonsingular.
  See Lemma~\ref{lemma:nontrivial-self-blanchfield} in the appendix.
\end{remark}

\begin{remark}
  We give a generalized version of Theorem~\ref{theorem:main} concerning compositions of distinct satellite operators at the end of Section~\ref{subsection:Blanchfield-bordism}.
  See Theorem~\ref{theorem:main-generalization}.
\end{remark}

\subsection{Proof of Theorem~\ref{theorem:main}}\label{subsection:proof-of-Theorem-A}

To prove Theorem~\ref{theorem:main}, we use the following lemma, which says that $P^n(K)$ is integrally $n$-solvable {\it modulo $P^n(U)$}. (Compare it with Lemma~\ref{lemma:construction-of-n-solvable-knots}.)

\begin{lemma}\label{lemma:affine-inclusion}
Let $P\colon \cC\to \cC$ be a winding number zero satellite operator. Then, for each $n\ge 1 $ and any knot $K$, $P^n(K)\#-P^n(U)$ is integrally $n$-solvable. In particular,  $\langle P^n(\cC)\rangle \subset \cF_n + \langle P^n(U)\rangle$.
\end{lemma}
\begin{proof} 
Consider the satellite operator $P^n\#(-P^n(U))\colon \cC\to \cC$ defined by
\[
\bigl(P^n\#(-P^n(U))\bigr)(K)=P^n(K)\#-P^n(U).
\]

Let $\eta$ be the axis of $P^n\#(-P^n(U))$. Since $P$ has winding number zero, one readily sees (e.g., \cite[Lemma~4.9]{Cha:2010-01}) that
\[
\eta\in \pi_1(S^3 \sm (P^n(U)\#-P^n(U)))^{(n)}.
\]
Since $P^n(U)\#-P^n(U)$ is slice,
\[
  P^n(K)\#-P^n(U) = (P^n\#-P^n(U))(K)\in \cF_n
\]
by Lemma~\ref{lemma:construction-of-n-solvable-knots}~(1).
\end{proof}

By Lemma~\ref{lemma:affine-inclusion}, Theorem~\ref{theorem:main} holds if $\langle P^n(\cC)\rangle$ has infinite rank in $\cC/(\cF_{n+1}+\langle P^{n+1}(U)\rangle)$. 
Note that for an abelian group $G$ and an arbitrary element $a\in G$, $G$ has infinite rank if and only if $G/\langle a\rangle$ has infinite rank. Therefore, it suffices to show the following.

\begin{proposition}\label{proposition:infinite-rank-iteration}
Suppose that $P\colon \cC\to \cC$ is a winding number zero satellite operator with axis $\eta$ satisfying $\Bl(\eta,\eta)\ne 0$. Then, for each $n\ge 1$, $\langle P^n(\cC)\rangle$ has infinite rank in $\cC/\cF_{n+1}$. 
\end{proposition}

The rest of this section is devoted to proving Proposition~\ref{proposition:infinite-rank-iteration}.

\begin{proof}[Proof of Proposition~\ref{proposition:infinite-rank-iteration}]
For each integer $m\ge 1$, we will show that there exist $m$ knots $J_0^1$,~\dots, $J_0^m$ such that $P^n(J_0^1)$,~\dots, $P^n(J_0^m)$ are linearly independent in $\cC/\cF_{n+1}$.
It implies that $\rank \langle P^n(\cC)\rangle \ge m$ in $\cC/\cF_{n+1}$ for all~$m$, and therefore $\langle P^n(\cC)\rangle$ has infinite rank in $\cC/\cF_{n+1}$.

Fix $n$ and~$m$. 
Write $\Bl(\eta,\eta)=f(t)/g(t)\in \Q(t)/\Z[t^{\pm 1}]$, where $f(t)$ and $g(t)$ are Laurent polynomials in $\Z[t^{\pm 1}]$ with $\deg f(t) < \deg g(t)$, $f(t)\ne 0$.
For a prime $p$, let
\[
\Bl\nolimits^p\colon H_1(M(P(U));\Z_p[t^{\pm 1}])\times H_1(M(P(U));\Z_p[t^{\pm 1}])\to \Z_p(t)/\Z_p[t^{\pm 1}]
\] 
be the mod $p$ reduction of the Blanchfield form.
Then one readily sees that $\Bl^p(\eta,\eta)$ is nontrivial if $p$ is greater than the absolute values of all the coefficients of $f(t)$ and~$g(t)$.
Therefore, there exist distinct primes $p_1, \ldots, p_m$ such that $\Bl^{p_i}(\eta, \eta)$ is nontrivial, $1\le i\le m$. 

Choose a constant $L>0$ such that $|\rhot(M(P(U)), \phi)|< L$ for all~$\phi$, using Theorem~\ref{theorem:universal-bound}.
For $1\le i\le m$, let $\omega_i=\text{exp}(2\pi\sqrt{-1}/p_i) \in ~\C$.

\begin{lemma}\label{lemma:knots-J_0^i}
There exist knots $J_0^1, J_0^2, \ldots, J_0^m$ satisfying the following.
\begin{enumerate}
	\item $\sigma_{J_0^i}(\omega_i)=\sigma_{J_0^i}(\omega_i^{-1}) > \frac12 mnL$ for each $i$.
	\item $\sigma_{J_0^i}(\omega_i^r)=0$ for $r\ne \pm 1$ \textup{(mod $p_i$)} for each $i$.
	\item If $i\ne j$, then $\sigma_{J_0^i}(\omega_j^r)=0$ for all $r$.
\end{enumerate}
\end{lemma}
\begin{proof}
Note that the numbers $k/p_i$ are pairwise distinct for $1\le i\le m$, $1\le k\le p_i-1$.
By \cite[Lemma~5.6]{Cha:2009-1}, it follows that there exist knots $J^i$ $(1\le i\le m)$ that satisfy the conditions (2), (3) and that $\sigma_{J^i}(\omega_i)\ne 0$.
By changing the orientation of $J^i$ if necessary, we may assume that $\sigma_{J^i}(\omega_i)>0$.
Let $J_0^i$ be the connected sum of sufficiently many copies of $J^i$ so that $\sigma_{J_0^i}(\omega_i)>\frac12 mnL$.
\end{proof}

Let $J_k^i=P^k(J_0^i)$ for $1\le k\le n$ and $1\le i\le m$ where $J_0^i$ is the knot in Lemma~\ref{lemma:knots-J_0^i}. We will show that the $J_n^i$ $(1\le i\le m)$ are linearly independent in $\cC/\cF_{n+1}$.

Suppose not. Then there exists a finite linear combination $J=\#_{i=1}^m a_iJ_n^i$ lying in $\cF_{n+1}$ where $a_i\ne 0$ for some~$i$.  By rearranging $J_n^i$ and replacing $J$ with $-J$ if necessary, we may assume that for some $m'\le m$ we have $J=\#_{i =1}^{m'}a_iJ_n^i$ where $a_1\ge |a_i| > 0$ for each $1\le i\le m'$.

We construct a 4-manifold $W_0$ depicted in Figure~\ref{figure:infinite-rank} using the following 4-manifolds
(cf.\ \cite[p.~4792]{Cha:2010-01}).

\begin{enumerate}
	\item Since $J\in \cF_{n+1}$, there is an integral $(n+1)$-solution $V$ bounded by~$M(J)$.
  \item
  A ``standard cobordism'' $C = C(a_1J_n^1, \ldots, a_{m'}J_n^{m'})$ bounded by 
	\[
	\partial C(a_1J_n^1, \ldots, a_{m'}J_n^{m'}) = \biggl(\bigsqcup_{i=1}^{m'}a_iM(J_n^i)\biggr) \sqcup (-M(J))
	\]
	is described in \cite[p.~113]{Cochran-Orr-Teichner:2004-1}.
  The fact we need is that the inclusion-induced map $H_2(\partial C;\Z[G]) \to H_2(C;\Z[G])$ is surjective for any group $G$, and hence the $L^2$-signature ${\sign^{(2)}_G} C$ vanishes for any homomorphism $\pi_1(C) \to G$ (see \cite[Lemma~2.4]{Cochran-Harvey-Leidy:2009-01} and its proof).
  We note that the special case of $G=\{e\}$ gives $\sign C=0$. 
  \item For $0\le k\le n-1$, a cobordism $E(P,J_k^i)$ bounded by 
  \[
  \partial E(P,J_k^i)= M(J_k^i)\sqcup M(P(U)) \sqcup (-M(J_{k+1}^i)).
  \]
  is described in~\cite[p.~4792]{Cha:2010-01} (see the construction of $E_k$ therein) and the proof of \cite[Lemma~2.3]{Cochran-Harvey-Leidy:2009-01}.
  For brevity, let $E_k^i = E(P,J_k^i)$.
  The fact we need is that the inclusion-induced map $H_2(\partial E_k^i;\Z[G]) \to H_2(E_k^i;\Z[G])$ is surjective for any group $G$ (see \cite[Lemma~2.4]{Cochran-Harvey-Leidy:2009-01} and its proof) and hence $\sign^{(2)}_GE_k^i=0$ for any $\pi_1(E_k^i) \to G$.	
\end{enumerate}

Note that $M(P(U))$ is a boundary component of each~$E_k^i$.
For clarity, denote $M(P(U))\subset \partial E_k^i$ by $M_k^i(P(U))$.
Also, denote the axis $\eta \subset M_k^i(P(U))$ of $P$ by~$\eta_k^i$.

Now we define 4-manifolds $W_0, W_1, \ldots, W_n$ inductively in the reverse order as follows.
Let $W_n=V\cup_{M(J)} C$. Then $\partial W_n=\bigsqcup_ia_iM(J_n^i)$. 

Suppose $0\le k\le n-1$ and that $W_{k+1}$ has been constructed so that
\[
\partial W_{k+1}= \bigsqcup_{i=1}^{m'}a_i\Biggl( M(J_{k+1}^i)\sqcup\Biggl(\bigsqcup_{j=k+1}^{n-1}M_{j}^i(P(U))\Biggr)\Biggr).
\]
By glueing $a_i E^i_k$ to $W_{k+1}$ along the common boundary components $a_i M(J_{k+1}^i)$ for $i=1,\ldots,m'$, define a 4-manifold $W_k$ as follows: 
\[
W_k=W_{k+1}
\cup
\Biggl(\bigsqcup_{i=1}^{m'} a_iE_k^i\Biggr).
\]
Then
\[
\partial W_k= \bigsqcup_{i=1}^{m'}a_i\Biggl( M(J_k^i)\sqcup\Biggl(\bigsqcup_{j=k}^{n-1}M_{j}^i(P(U))\Biggr)\Biggr).
\]
When $k=0$, we have
\begin{align*}
W_0&=V\cup C\cup \Biggl(\bigsqcup_{i=1}^{m'} a_iE_{n-1}^i\Biggr)\cup \Biggl(\bigsqcup_{i=1}^{m'} a_iE_{n-2}^i\Biggr)\cup \cdots \cup \Biggl(\bigsqcup_{i=1}^{m'} a_iE_0^i\Biggr), \text{ and } 
\\
\partial W_0&= \bigsqcup_{i=1}^{m'}a_i\Biggl( M(J_0^i)\sqcup\Biggl(\bigsqcup_{j=0}^{n-1}M_j^i(P(U))\Biggr)\Biggr).
\end{align*}
See Figure~\ref{figure:infinite-rank} for a schematic diagram.

\begin{figure}[t]
  \includestandalone[scale=0.9]{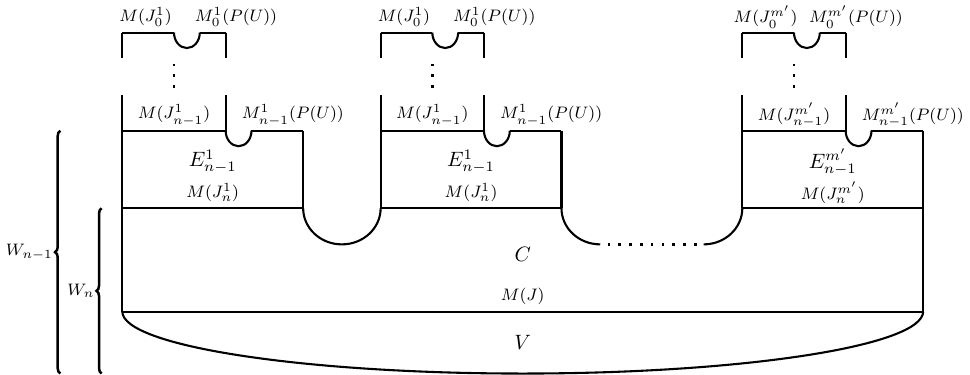}
  \caption{The cobordism $W_0$}
  \label{figure:infinite-rank}
\end{figure}

Let $\phi\colon \pi_1(W_0)\to G$ be a homomorphism. We will make a specific choice of $\phi$ later.
For brevity, we denote by $\phi$ restrictions of $\phi$ to subspaces of~$W_0$.

From the definition of $\rhot$-invariants (see Section~\ref{section:preliminary}), we have 
\begin{equation}\label{equation:rho-and-signature}
\rhot(\partial W_0, \phi) = \sign_G^{(2)}(W_0)-\sign(W_0).
\end{equation}
Using two different computations, we will show that $\rhot(\partial W_0, \phi) >0$ and $\sign_G^{(2)}(W_0)-\sign(W_0)=0$ for some choice of~$\phi$.
This contradicts~\eqref{equation:rho-and-signature}, so it completes the proof.

We compute $\sign_G^{(2)}(W_0)-\sign(W_0)$. For brevity, for a 4-manifold $X$, let 
\[
S_G(X)=\sign_G^{(2)}(X)-\sign(X).
\] 
Let $\epsilon_i=a_i/|a_i|$ and $E_k^{i,j}$ be the $j$th copy of $E_k^i$, $1\le j\le |a_i|$, used in the construction of~$W_0$.
By Novikov additivity, 
\[
S_G(W_0) = S_G(V) + S_G(C) + \sum_{k=0}^{n-1}\sum_{i=1}^{m'} \sum_{j=1}^{|a_i|}S_G(\epsilon_i E_k^{i,j}).
\]

By (2) and (3) above, $S_G(C)=0$ and $S_G(\epsilon_iE_k^{i,j})=0$ for all $i$, $j$ and~$k$. 
Since $V$ is an integral $(n+1)$-solution with $\partial V=M(J)$, Theorem~\ref{theorem:vanishing-rho-invariants} ensures that $S_G(V)=\rhot(M(J),\phi)=0$ whenever $\phi\colon \pi_1(W_0)\to G$ satisfies the condition (H1) below. 
\begin{itemize}
\item[(H1)] The codomain $G$ lies in Strebel's class $D(\Z_{p_1})$, $G^{(n+1)}=\{e\}$, and $\phi$ sends a meridian of $J$ to an element of infinite order in~$G$.
\end{itemize}	
Therefore, it follows that $S_G(W_0)=0$ for $\phi$ satisfying~(H1).

Next, we compute $\rhot(\partial W_0, \phi)$. Let $M(J_k^i)^j$ and $M_k^{i,j}$  denote the copies of $M(J_k^i)$ and $M_k^i(P(U))$ that appear in the top boundary of $E_k^{i,j}$ for $1\le j\le |a_i|$, respectively. That is, 
\[
\partial E_k^{i,j}=M(J_k^i)^j\sqcup M_k^{i,j}\sqcup (-M(J_{k+1}^i)^j).
\] 
Then
\[
\partial W_0= \Biggl(\bigsqcup_{i=1}^{m'} \bigsqcup_{j=1}^{|a_i|}\epsilon_iM(J_0^i)^j\Biggr) \sqcup \Biggl(\bigsqcup_{k=0}^{n-1}\bigsqcup_{i=1}^{m'} \bigsqcup_{j=1}^{|a_i|}\epsilon_iM_k^{i,j}\Biggr).
\]
By additivity of $\rhot$-invariants and Lemma~\ref{lemma:properties-of-rho-invariants}~(3), we have
\begin{equation}\label{equaiton:rho-invariant-of-W_0}
\rhot(\partial W_0, \phi) = \sum_{i=1}^{m'}\sum_{j=1}^{|a_i|}\epsilon_i\rhot(M(J_0^i)^j,\phi) + \sum_{k=0}^{n-1} \sum_{i=1}^{m'} \sum_{j=1}^{|a_i|} \epsilon_i\rhot(M_k^{i,j},\phi). 
\end{equation}

Now suppose that $\phi:\pi_1(W_0)\to G$ satisfies the condition (H2) below.
\begin{itemize}
\item[(H2)] The image of $\phi\colon \pi_1(M(J_0^i)^j)\to G$ is isomorphic to~$\Z_{p_1}$ for each $i$ and~$j$.
\end{itemize}
 
Recall that $\omega_1=\text{exp}(2\pi\sqrt{-1}/p_1)$. By Lemma~\ref{lemma:properties-of-rho-invariants}~(5), 
\begin{equation*}
  \rhot(M(J_0^i)^j,\phi)=\frac{1}{p_1}\sum_{r=0}^{p_1-1}\sigma_{J_0^i}(\omega_1^r)
\end{equation*}
for each $i$ and~$j$.
So, by our choice of the $J_0^i$ (see Lemma~\ref{lemma:knots-J_0^i}),
\begin{equation}
  \rhot(M(J_0^i)^j, \phi)= \left\{
    \begin{array}{ll}
      2\sigma_{J_0^1}(\omega_1) > mnL & \text{ if } i=1,\\
      0 & \text{ otherwise.}
    \end{array}
  \right.
  \label{equation:top-level-rho}  
\end{equation}
Now we have
\begin{equation*}
    \begin{aligned}
    \rhot(\partial W_0,\phi)	&> a_1mnL + \sum_{k=0}^{n-1} \sum_{i=1}^{m'} \sum_{j=1}^{|a_i|} \epsilon_i\rhot(M_k^{i,j},\phi)
      && \text{by \eqref{equation:top-level-rho}}\\
    & >  a_1 mn L -nL\sum_{i=1}^{m'}|a_i|
      && \text{since $|\rhot(M_k^{i,j},\phi)|< L$}\\
        &= a_1 mn L -a_1m'nL
      && \text{since $|a_i| \le a_1$}\\
    & \ge a_1 mn L- a_1mn L = 0
      && \text{since $m'\le m$.}
  \end{aligned}
\end{equation*}
 
So, $\rhot(\partial W_0,\phi) \ne S_G(W_0)$.
It contradicts~\eqref{equation:rho-and-signature}.
Therefore, it completes the proof, modulo the proof of Proposition~\ref{proposition:computation-of-rho-invariant-of-J_0^i} stated below.
\end{proof}

\begin{proposition}\label{proposition:computation-of-rho-invariant-of-J_0^i}
There exists a homomorphism $\phi\colon \pi_1(W_0)\to G$ that satisfies both (H1) and (H2).
\end{proposition}

Our assumption that $\Bl(\eta,\eta)$ is nontrivial will be used to prove Proposition~\ref{proposition:computation-of-rho-invariant-of-J_0^i}.
We will give a proof of Proposition~\ref{proposition:computation-of-rho-invariant-of-J_0^i} in Section~\ref{subsection:proof-of-lemma} below.

\begin{remark}\label{remark:main-generalization}
  The arguments in the proof of Proposition~\ref{proposition:infinite-rank-iteration} work only for a finite family of knots $J_1^i$,~\dots, $J_m^i$ and cannot be applied to an infinite family of knots: in Lemma~\ref{lemma:knots-J_0^i}~(1), the choice of knots $J_0^i$, $1\le i\le m$, depends on the finiteness of the cardinality~$m$.
  \end{remark}

\subsection{Proof of Proposition~\ref{proposition:computation-of-rho-invariant-of-J_0^i}}\label{subsection:proof-of-lemma}

In this subsection, we give a proof of Proposition~\ref{proposition:computation-of-rho-invariant-of-J_0^i}, which states that there exists a homomorphism $\phi\colon \pi_1(W_0)\to G$ that satisfies both (H1) and (H2). For the reader's convenience, we recall (H1) and (H2) below.

\begin{itemize}
\item[(H1)] The codomain $G$ lies in Strebel's class $D(\Z_{p_1})$, $G^{(n+1)}=\{e\}$, and $\phi$ sends a meridian of $J$ to an element of infinite order in~$G$.
\item[(H2)] The image of $\phi\colon \pi_1(M(J_0^i)^j)\to G$ is isomorphic to~$\Z_{p_1}$ for each $i$ and~$j$.
\end{itemize}	

To obtain $\phi$ satisfying (H1) and (H2), we define subgroups $\cP^i\Gamma$ of a group $\Gamma$ for $i=0,1,\ldots, n+1$ as below, following \cite[Definition~4.1]{Cha:2010-01}. 
Let $\cP^0\Gamma = \Gamma$. For $i=0,1,\ldots,n-1$, define inductively
\[
\cP^{i+1}\Gamma = \Ker\left\{\cP^i\Gamma \to \frac{\cP^i\Gamma}{[\cP^i\Gamma,\cP^i\Gamma]}\to \frac{\cP^i\Gamma}{[\cP^i\Gamma,\cP^i\Gamma]}\otimes_\Z \Q = H_1\big(\Gamma;\Q[\Gamma/\cP^i\Gamma]\big)\right\}.
\]
Finally, define 
\[
\cP^{n+1}\Gamma = \Ker\left\{\cP^n\Gamma \to \frac{\cP^n\Gamma}{[\cP^n\Gamma,\cP^n\Gamma]}\to \frac{\cP^n\Gamma}{[\cP^n\Gamma,\cP^n\Gamma]}\otimes_\Z \Z_{p_1} = H_1\big(\Gamma;\Z_{p_1}[\Gamma/\cP^n\Gamma]\big)\right\}.
\] 
Each $\cP^i\Gamma$ is a normal subgroup of~$\Gamma$.
We will apply this definition to $\Gamma=\pi_1(W_k)$, $k\ge 0$.

Let $G=\pi_1(W_0)/\cP^{n+1}\pi_1(W_0)$.
Let $\phi\colon \pi_1(W_0)\to G$ be the quotient map. 

\begin{proof}[Proof of (H1)]
  One can readily see that $G^{(n+1)}=\{e\}$ since $\pi_1(W_0)^{(n+1)}\subset \cP^{n+1}\pi_1(W_0)$. 
  For $\Gamma=\pi_1(W_0)$ and $0\le k\le n$, $\cP^k\Gamma/\cP^{k+1}\Gamma$ injects into $H_1\left(\Gamma;R\left[\Gamma/\cP^k\Gamma\right]\right)$ with $R=\Q$ or $\Z_{p_1}$, which is an $R$-vector space.
    So, by Lemma~\ref{lemma:amenable-D(R)-group}(2), $G$ lies in Strebel's class $D(\Z_{p_1})$ (see also \cite[Lemma~4.3]{Cha:2010-01}).

    Also, a meridian of $J$ generates
  \[
    \pi_1(W_0)/\cP^1  \pi_1(W_0) \cong H_1(W_0)/\text{torsion} = H_1(W_0) \cong H_1(M(J)) = \Z    
  \]
  onto which $G$ surjects.
  So $\phi$ sends a meridian of $J$ to an infinite order element in~$G$.  
    \def\qedsymbol{} 
\end{proof}

\begin{proof}[Proof of (H2)]
Let $\mu_k^{i,j}$ be a meridian of $M(J_k^i)^j$. Since $H_1(M(J_0^i)^j)\cong~\Z$ is generated by $\mu_0^{i,j}$, and since the group $\cP^n\pi_1(W_0)/\cP^{n+1}\pi_1(W_0)$ is a $\Z_{p_1}$-vector space, $\phi$ satisfies (H2) if it satisfies (H2$'$) below.
\begin{itemize}
  \item[(H2$'$)] The image of each $\mu_0^{i,j}$ under $\phi$ is nontrivial in $\cP^n\pi_1(W_0)/\cP^{n+1}\pi_1(W_0)$.
\end{itemize}

For brevity, we write $\mu_k$ for $\mu_k^{i,j}$ when the superscripts are understood. We will show, inductively, that for $k=n, n-1, \ldots, 0$, $\mu_k=\mu_k^{i,j}$ is nontrivial in $\cP^{n-k}\pi_1(W_k)/\cP^{n-k+1}\pi_1(W_k)$.
The $k=0$ case implies the desired conclusion~(H2$'$).

For $k=n$, the claim follows from
\[
\cP^0\pi_1(W_n)/\cP^1\pi_1(W_n)\cong H_1(W_n) \cong H_1(M(J_n^i)^j)\cong \langle \mu_n\rangle \cong \Z.
\]

Suppose that $\mu_{k+1}=\mu_{k+1}^{i,j}$  is nontrivial  in $\cP^{n-k-1}\pi_1(W_{k+1})/\cP^{n-k}\pi_1(W_{k+1})$, $k\le n-1$.
We will show that $\mu_k$ is nontrivial in the quotient $\cP^{n-k}\pi_1(W_k)/\cP^{n-k+1}\pi_1(W_k)$.

First, we show that $\mu_k\in \cP^{n-k}\pi_1(W_k)$.
Recall that $M_k^{i,j} \subset \partial E_k^{i,j}$ is a copy of~$M(P(U))$.
Let $\eta_k=\eta_k^{i,j}\subset M_k^{i,j}$ be the corresponding copy of the axis of~$P$.
Since $\eta_k$ is identified with $\mu_k$ in $M(J_{k+1}^i)^j$ and $P$ has winding number zero,  we have $\mu_k\in \langle \mu_{k+1}\rangle^{(1)}$, where $\langle - \rangle$ denotes the subgroup of $\pi_1(W_k)$ normally generated by~$-$.
Inductively, it follows that $\mu_k\in \langle \mu_n\rangle^{(n-k)}$. Therefore, $\mu_k\in \pi_1(W_k)^{(n-k)} \subset \cP^{n-k}\pi_1(W_k)$. (See also the last paragraph of \cite[p.~4799]{Cha:2010-01}.)

Next, we assert that 
\begin{equation}\label{equation:quotient-isomorphism}
\cP^{n-k}\pi_1(W_k)/\cP^{n-k+1}\pi_1(W_k)\cong \cP^{n-k}\pi_1(W_{k+1})/\cP^{n-k+1}\pi_1(W_{k+1}).
\end{equation}
By arguments in \cite[Assertion~1 on p.~4799]{Cha:2010-01}, we have 
\begin{equation}\label{equation:quotient-isomorphism-2}
 \pi_1(W_k)/\pi_1(W_k)^{(n-k+1)}\cong \pi_1(W_{k+1})/\pi_1(W_{k+1})^{(n-k+1)}.
 \end{equation}
Also, from the definitions of $\cP^i\Gamma$, one can see that if a homomorphism $A\to B$ induces $A/A^{(j)}\cong B/B^{(j)}$, then the homomorphism also induces $\cP^i A\cong \cP^i B$ for $i\le j$.
Therefore, the inclusion $W_{k+1}\hookrightarrow W_k$ induces $\cP^i\pi_1(W_k)\cong \cP^i\pi_1(W_{k+1})$ for $i\le n-k+1$, and \eqref{equation:quotient-isomorphism} follows.

By \eqref{equation:quotient-isomorphism}, we only need to show that $\mu_k$ is nontrivial in $\cP^{n-k}\pi_1(W_{k+1})/\cP^{n-k+1}\pi_1(W_{k+1})$. Let $\pi=\pi_1(W_{k+1})$ and $\Gamma=\pi/\cP^{n-k}\pi$.
Let $R=\Q$ when $k>0$ and $R=\Z_{p_1}$ when $k=0$.
Since $\cP^{n-k}\pi/\cP^{n-k+1}\pi$  injects into $H_1(W_{k+1};R[\Gamma])$, it suffices to show that $\mu_k$ is nontrivial in $H_1(W_{k+1};R[\Gamma])$.

To show this, we use the (higher order) Blanchfield forms, following the ideas in~\cite{Cochran-Harvey-Leidy:2009-01} and~\cite{Cha:2010-01}.
Denote the Blanchfield form of $M(J_{k+1}^i)^j$ over~$R[t^{\pm1}]$ by
\[
\Bl\nolimits^R\colon H_1(M(J_{k+1}^i)^j;R[t^{\pm 1}]) \times H_1(M(J_{k+1}^i)^j;R[t^{\pm 1}]) \to R(t)/R[t^{\pm 1 }].
\]
Also, let
\[
\cB\colon H_1(M(J_{k+1}^i)^j;R[\Gamma])\times H_1(M(J_{k+1}^i)^j;R[\Gamma])\to  \cK/R[\Gamma]
\]
be the higher order Blanchfield form of $M(J_{k+1}^i)^j$ over~$R[\Gamma]$ defined in \cite[Theorem~2.13]{Cochran-Orr-Teichner:2003-1}.
Here, $\cK$ is the (skew) field of quotients of $R[\Gamma]$:
since $\cP^i\pi/\cP^{i+1}\pi$ is torsion-free for $0\le i \le n-1$ by the definition of $\cP^{i+1}\pi$ and since the range of $k$ is $0,\ldots,n-1$, the group $\Gamma = \pi/\cP^{n-k}\pi$ is poly-torsion-free-abelian.
(If $k$ were $-1$, $\pi/\cP^{n+1}\pi$ would be allowed to have $p$-torsion.)
It follows that $R[\Gamma]$ embeds into its (skew) field of quotients~$\cK$, due to \cite[Proposition~2.5]{Cochran-Orr-Teichner:2003-1} when $R=\Q$ (i.e.\ $k>0$) and \cite[Lemma~5.2]{Cha:2010-01} when $R=\Z_{p_1}$ (i.e.\ $k=0$).



We assert that $\cB(\mu_k,\mu_k)\ne0$. By our induction hypothesis that the meridian $\mu_{k+1}$ is nontrivial in the torsion-free abelian group $\cP^{n-k-1}\pi/\cP^{n-k}\pi\subset \Gamma$, one can see that 
\begin{equation}\label{equation:isomorphism-Alexander-modules}
H_1(M(J_{k+1}^i)^j;R[\Gamma])\cong R[\Gamma]\otimes_{R[t^{\pm1 }]} H_1(M(J_{k+1}^i)^j;R[t^{\pm 1}]).
\end{equation}
By \cite[Theorem~4.7]{Leidy:2006-1} and \cite[Theorem~5.4]{Cha:2010-01}, we have $\cB(1\otimes x,1\otimes x)=0$ if and only if $\Bl^R(x,x)=0$.
Also by our assumption on $\eta_k$ and choice of $p_1$, we have $\Bl^R(\eta_k,\eta_k)\ne 0$.
Since $\eta_k$ is identified with $\mu_k$ in $M(J_{k+1}^i)^j$, it follows that $\Bl^R(\mu_k,\mu_k)\ne 0$, and hence $\cB(1\otimes \mu_k,1\otimes \mu_k)\ne 0$.
Therefore, viewing $\mu_k=1\otimes \mu_k$ as an element in $H_1(M(J_{k+1}^i)^j;R[\Gamma])$ via the isomorphism~\eqref{equation:isomorphism-Alexander-modules}, it follows that $\cB(\mu_k,\mu_k)\ne 0$.

Now we invoke the following lemma, which will be proved in Section~\ref{subsection:Blanchfield-bordism}.

\begin{lemma}\label{lemma:Blanchfield-bordism}
The pair $(W_{k+1}, \phi\colon \pi_1(W_{k+1})\to \Gamma)$ is an $R$-coefficient Blanchfield bordism.
\end{lemma}

Here, we use the notion of an \emph{$R$-coefficient Blanchfield bordism} in the sense of \cite[Definition~4.11]{Cha:2014-1}.
The only property of a Blanchfield bordism we need to complete the proof of Proposition~\ref{proposition:computation-of-rho-invariant-of-J_0^i} is that for the kernel
\[
M=\Ker \{H_1(M(J_{k+1}^i)^j;R[\Gamma])\rightarrow H_1(W_{k+1};R[\Gamma])\}
\] 
of the inclusion-induced map, we have $\cB(x,x)=0$ for all $x\in M$ (see \cite[Theorem~4.12]{Cha:2014-1}). 

From this, it follows that $\mu_k\notin M$ since $\cB(\mu_k,\mu_k)\ne 0$. Therefore, $\mu_k$ is nontrivial in $H_1(W_{k+1};R[\Gamma])$. This completes the proof of Proposition~\ref{proposition:computation-of-rho-invariant-of-J_0^i} modulo the proof of Lemma~\ref{lemma:Blanchfield-bordism}, which is given below.
  \def\qedsymbol{} 
\end{proof}

\subsection{Lagrangian bordism and Blanchfield bordism}\label{subsection:Blanchfield-bordism}
In this subsection, we give a proof of Lemma~\ref{lemma:Blanchfield-bordism}. 
We also discuss a generalization of Theorem~\ref{theorem:main} at the end of the subsection. 

Throughout this subsection, we suppose the following: let $R=\Z_p$ or $\Q$. Let $W$ be a compact connected 4-manifold with boundary and $\psi\colon \pi_1(W)\to G$ be a homomorphism where $G$ is a poly-torsion-free-abelian group. Let $\cK G$ be the (skew) field of quotients of $RG$. Let $\lambda_G$ be the equivariant intersection form 
\[
\lambda_G\colon H_2(W;RG)\times H_2(W;RG)\to RG.
\]

The following definition is inspired by the notion of an $n$-bordism in the sense of \cite[Definition~5.2]{Cochran-Harvey-Leidy:2009-01} and \cite[Definition~5.5]{Cha:2010-01} and by \cite[Theorem~4.13]{Cha:2014-1}.

\begin{definition}\label{definition:lagrangian-bordism}
A pair $(W,\psi)$ is an \emph{$R$-coefficient Lagrangian bordism} if it satisfies the following.
\begin{enumerate}
	\item There exist elements $\ell_1,\ldots, \ell_m$, $d_1,\ldots, d_m \in H_2(W;RG)$ such that $\lambda_G(\ell_i,\ell_j)=0$ and $\lambda_G(\ell_i,d_j)=\delta_{ij}$, $1\le i,j\le m$. 
	\item $\rank_R \Coker\{H_2(\partial W;R)\to H_2(W;R)\}\le 2m$. 
\end{enumerate}
\end{definition}

If $W$ is an integral $n$-solution, then $H_2(\partial W;R)\to H_2(W;R)$ is trivial. Therefore, the lemma below immediately follows from the definition of an integral $n$-solution (see \cite[Definition~3.1]{Cha:2010-01}).

\begin{lemma}\label{lemma:integral-solution-lagrangian-bordism}
Suppose that $W$ is an integral $n$-solution. Then, for any homomorphism $\psi\colon \pi_1(W)\to G$ that factors through $\pi_1(W)/\pi_1(W)^{(n)}$, the pair $(W,\psi)$ is an $R$-coefficient Lagrangian bordism.
\end{lemma}

Recall that Lemma~\ref{lemma:Blanchfield-bordism} asserts that $(W_{k+1},\phi)$ is an $R$-coefficient Blanchfield bordism.
In the proof of Lemma~\ref{lemma:Blanchfield-bordism}, we do not directly verify that $(W_{k+1},\phi)$ satisfies the definition of an $R$-coefficient Blanchfield bordism given in~\cite[Definition~4.11]{Cha:2014-1} but use the following fact.

\begin{proposition}\label{proposition:lagrangian-Blanchfield-bordism} 
Suppose that $\rank_R H_1(M;R) = 1$ for each component $M$ of $\partial W$. If $(W,\psi)$ is an $R$-coefficient Lagrangian bordism, then it is an $R$-coefficient Blanchfield bordism.
\end{proposition}
\begin{proof}
  By \cite[Lemma~4.15]{Cha:2014-1}, $(W,\psi)$ is an $R$-coefficient Blanchfield bordism if the following hold.
  \begin{enumerate}
    \item There exist elements $\ell_1,\ldots, \ell_m$, $d_1,\ldots, d_m \in H_2(W;RG)$ such that $\lambda_G(\ell_i,\ell_j)=0$ and $\lambda_G(\ell_i,d_j)=\delta_{ij}$, $1\le i,j\le m$. 
    \item $\dim_{\cK G} \Coker \{H_2(\partial W;\cK G)\rightarrow H_2(W;\cK G)\}\le 2m$.
  \end{enumerate}
  Definition~\ref{definition:lagrangian-bordism}~(1) implies~(1).
  Definition~\ref{definition:lagrangian-bordism}~(2) implies (2) by Lemma~\ref{lemma:rank-comparison} below.
\end{proof}


\begin{lemma}
\label{lemma:rank-comparison}
Suppose $\rank_R H_1(M;R) = 1$ for each component $M$ of $\partial W$. Then,
\[
\dim_{\cK G} \Coker \{H_2(\partial W;\cK G)\rightarrow H_2(W;\cK G)\}\le \rank_R \Coker\{H_2(\partial W;R)\rightarrow H_2(W;R)\}.
\]
\end{lemma}
\begin{proof}
  In \cite[Lemma~5.10~(1)]{Cochran-Harvey-Leidy:2009-01}, it was shown that the equality 
  \[
    \dim_{\cK G} \Coker \{H_2(\partial W;\cK G)\rightarrow H_2(W;\cK G)\} = \rank_R \Coker\{H_2(\partial W;R)\rightarrow H_2(W;R)\}.
  \] 
  holds under a stronger hypothesis.
  Examining the proof of~\cite[Lemma~5.10~(1)]{Cochran-Harvey-Leidy:2009-01}, which analyzes the long exact sequences of $(W,\partial W)$ with coefficients in $R$ and $\cK G$, one sees that the same argument shows the desired inequality under our weaker hypothesis.
  For the reader's convenience, we describe details below.
  We remark that we do not use that $W$ is an $R$-coefficient Lagrangian bordism in this proof.

  If $\psi\colon \pi_1(W)\to G$ is trivial, then one can readily obtain the desired inequality (actually equality in this case). Suppose that $\psi$ is nontrivial.
  For brevity, let 
  \begin{equation*}
      \begin{aligned}
      r &=\rank_R \Coker\{H_2(\partial W;R)\rightarrow H_2(W;R)\} \\
      r' &=\dim_{\cK G} \Coker \{H_2(\partial W;\cK G)\rightarrow H_2(W;\cK G)\}
      \end{aligned}
  \end{equation*}
  and denote by $\beta_i$ and $b_i$ be the $i$th Betti numbers with $R$-coefficients and $\cK G$-coefficients, respectively.
  Using the homology long exact sequence for $(W,\partial W)$ with $R$-coefficients and the Poincar\'{e} duality $\beta_i(W)=\beta_{4-i}(W,\partial W)$, one can readily see that 
  \[
  r=\chi(W)+2\beta_1(W)-2-\beta_1(\partial W)+\beta_0(\partial W),
  \]
  where $\chi$ is the Euler characteristic.
  For the $\cK G$-coefficient case, since $\psi$ is nontrivial, $b_0(W)=0$ by \cite[Propositions~2.9]{Cochran-Orr-Teichner:2003-1}, and the the same argument gives 
  \[
  r'=\chi(W)+2b_1(W) - b_1(\partial W)+b_0(\partial W).
  \]
  Therefore,
  \[
  r-r' = 2(\beta_1(W)-1-b_1(W)) - (\beta_1(\partial W) - b_1(\partial W)) + (\beta_0(\partial W) - b_0(\partial W)).
  \]
  Since $\psi$ is nontrivial, we have $\beta_1(W)-1-b_1(W)\ge 0$ by \cite[Proposition~2.11]{Cochran-Orr-Teichner:2003-1}.
  So, to complete the proof, it suffices to show that for each component $M$ of $\partial W$
  \[
    \beta_1(M) - b_1(M) = \beta_0(M) - b_0(M).
  \]
  It holds if $\psi$ induces a trivial homomorphism of $\pi_1(M)$, since $\beta_i(M)=b_i(M)$ in this case.
  Otherwise, our hypothesis $\beta_1(M) = 1$ implies that $b_1(M)=0$ by \cite[Propositions~2.11]{Cochran-Orr-Teichner:2003-1}.
  Also $b_0(M)=0$ by \cite[Propositions~2.9]{Cochran-Orr-Teichner:2003-1}.
  Since $\beta_0(M) = 1$, the promised equality holds.
\end{proof}

In the following lemma, the main examples to keep in mind are $W=V$ and $X=C$ or $E^i_k$ in the proof of Theorem~\ref{theorem:main}.

\begin{lemma}\label{lemma:lagrangian-bordism-extension}
Let $W$ and $X$ be compact connected 4-manifolds with boundary. Suppose $\partial X= \partial_1 X\sqcup \partial_2 X$ where each $\partial_i X$ is a union of components of $\partial X$ and $W\cap X = \partial_1 X$. Let $W'=W\cup X$ and $\psi\colon \pi_1(W')\to G$ be a homomorphism. Suppose the following:
\begin{enumerate}
\item the inclusion induces an injection $H_i(\partial_1 X;R)\to H_i(X;R)$ for $i=1,2$, and
\item induces a surjection $H_2(\partial_2 X;R)\to H_2(X;R)$.
\end{enumerate}
If additionally $(W,\psi|_{\pi_1(W)})$ is an $R$-coefficient Lagrangian bordism, $(W',\psi)$ is an $R$-coefficient Lagrangian bordism.
\end{lemma}
\begin{proof}
If $\ell_1,\ldots, \ell_m$, $d_1,\ldots, d_m$ are elements in $H_2(W;RG)$ such that $\lambda_G(\ell_i,\ell_j)=0$ and $\lambda_G(\ell_i,d_j)=\delta_{ij}$, then they satisfy the same intersection form conditions on $H_2(W';RG)$. 

By (1) and (2) and an application of the Mayer--Vietoris sequence associated with $W'=W\cup X$, one can readily see that 
\[
\rank_R \Coker\{H_2(\partial W;R)\rightarrow H_2(W;R)\}\ge \rank_R \Coker\{H_2(\partial W';R)\rightarrow H_2(W';R)\}.
\]
Therefore, $\rank_R \Coker\{H_2(\partial W';R)\to H_2(W';R)\}\le 2m$.
\end{proof}

Recall that for knots $K_1,\ldots, K_n$, there is a standard cobordism $C=C(a_1K_1,\ldots, a_nK_n)$ such that
\[
\partial_- C = M(a_1K_1\# \cdots \# a_nK_n) \quad\text{ and }\quad \partial_+ C = \bigsqcup_{i=1}^na_iM(K_i).
\]
Also recall that for a knot $K$ and a satellite operator $P$ with axis $\eta$, there is a cobordism
$E(P,K)$ such that 
\[
\partial_- E(P,K) = M(P(K)) \quad\text{ and }\quad \partial_+ E(P,K) = M(K)\sqcup M(P(U)).
\]
The following lemma is well known (e.g. see the first paragraph of \cite[p.~114]{Cochran-Orr-Teichner:2004-1} and \cite[Lemma~2.5(3)]{Cochran-Harvey-Leidy:2009-01}), and it follows from a direct computation of the relevant homology cobordism groups using relative chain complexes.

\begin{lemma}\label{lemma:C-and-E-injectivity}
Let $X = C$ or $E(P,K)$ as above. Then, the inclusions induce an injection $H_i(\partial_- X;R)\to H_i(X;R)$ for $i=1,2$, and a surjection $H_2(\partial_+ X;R)\to H_2(X;R)$.
\end{lemma}

Now we are ready to prove Lemma~\ref{lemma:Blanchfield-bordism}.

\begin{proof}[Proof of Lemma~\ref{lemma:Blanchfield-bordism}]
Let $\pi=\pi_1(W_{k+1})$ and $\Gamma=\pi/\cP^{n-k}\pi$.
Recall that Lemma~\ref{lemma:Blanchfield-bordism} states that the pair $(W_{k+1},\psi\colon \pi\to \Gamma)$ is an $R$-coefficient Blanchfield bordism. 

We have
\[
W_{k+1}=V\cup C\cup \left(\bigsqcup_{i=1}^{m'} a_iE_{n-1}^i\right)\cup \left(\bigsqcup_{i=1}^{m'} a_iE_{n-2}^i\right)\cup \cdots \cup \left(\bigsqcup_{i=1}^{m'} a_iE_{k+1}^i\right). 
\]
Since $V$ is an integral $(n+1)$-solution, it is also an integral $(n-k)$-solution. We have $\Gamma^{n-k}=\{e\}$ since $\pi^{n-k}\subset \cP^{n-k}\pi$. Therefore, it follows from Lemma~\ref{lemma:integral-solution-lagrangian-bordism} that $(V,\psi)$ is an $R$-coefficient Lagrangian bordism.  One can readily see that $(W_{k+1}, \psi)$ is an $R$-coefficient Lagrangian bordism by repeatedly applying Lemmas~\ref{lemma:lagrangian-bordism-extension} and \ref{lemma:C-and-E-injectivity} with $X=C$ or $X=E^i_\ell$, $1\le i\le m'$ and $k+1\le \ell\le n-1$. 

Now, $(W_{k+1},\psi)$ is an $R$-coefficient Blanchfield bordism by Proposition~\ref{proposition:lagrangian-Blanchfield-bordism} since we have $\rank_R H_1(M;R) = 1$ for each boundary component $M$ of $W_{k+1}$. 
\end{proof}

\subsubsection*{Generalization of Theorem~\ref{theorem:main}}
Theorem~~\ref{theorem:main} generalizes to compositions of distinct satellite operators. The proof of Theorem~~\ref{theorem:main} does not use specific properties of the pattern $P(U)$, and this leads us to the following theorem.

\begin{theorem}[{Generalization of Theorem~\ref{theorem:main}}]
\label{theorem:main-generalization}
Let $n\ge 1$. For $1\le k\le n$, let $P_k\colon \cC\to \cC$ be a winding number zero satellite operator with axis $\eta_k$. Let $\Bl_k$ be the Blanchfield form of $P_k(U)$. If $\Bl_k(\eta_k,\eta_k)$ is nontrivial for $1\le k\le n$, then for any winding number zero satellite operator $P_{n+1}\colon \cC\to \cC$, the quotient
\[
\langle (P_n\circ P_{n-1}\circ \cdots \circ P_1)(\cC)\rangle/\langle (P_{n+1}\circ P_n\circ \cdots \circ P_1)(\cC)\rangle\]
has infinite rank. 
\end{theorem}
\begin{proof}
The proof is almost the same as that of Theorem~\ref{theorem:main}, and therefore we describe essential modifications and observations only:
\begin{enumerate}[label=(\arabic*)]
	\item Following the proof of Lemma~\ref{lemma:affine-inclusion}, we can show that 
	\[
	\langle (P_{n+1}\circ P_n\circ \cdots \circ P_1)(\cC)\rangle \subset \cF_{n+1} + \langle (P_{n+1}\circ P_n\circ \cdots \circ P_1)(U)\rangle.
	\]
	\item When we choose the distinct primes $p_1, \ldots, p_m$ using the mod $p$ reduction of the Blanchfield form, we use the Blanchfield form of the pattern $P_1(U)$. 
	\item Let $L$ be a constant such that $|\rhot(M(P_k(U)),\phi)|< L$ for all $\phi$ and $1\le k\le n$.
	\item Let $J_k^i=P_k(J_{k-1}^i)$, $1\le k\le n$.
	\item For $0\le k\le n-1$, let $E_k^i=E(P_{k+1},J_k^i)$. \qedhere
\end{enumerate}
\end{proof}

\section{Satellite operators that are not homomorphisms}\label{section:non-homomorphism}
In this section we prove Theorem~\ref{theorem:non-homomorphism}, which asserts the following:
let $P\colon \cC \to \cC$ be a winding number zero satellite operator with axis~$\eta$. If $\Bl(\eta, \eta)\ne 0$, then $P^n\colon \cC\to \cC$ is not a homomorphism for all $n\ge 1$.

\subsection{Proof of Theorem~\ref{theorem:non-homomorphism}}\label{subsection:proof-of-theorem-nonhomomorphism}
Choose a constant $L>0$ such that 
\[
|\rhot(M(P(U)), \phi)| <L
\] 
for all homomorphisms $\phi\colon \pi_1(M(P(U)))\to~G$, by invoking Theorem~\ref{theorem:universal-bound}. 

Let $J$ be a knot such that
\[
\rhot(M(J),\epsilon)>3nL
\] 
where $\epsilon\colon \pi_1(M(J_1))\to \Z$ is the abelianization. We will show that $P^n(J\# J) \ne P^n(J)\# P^n(J)$ in $\cC$, i.e., the knot
\[
  K=P^n(J)\# P^n(J)\#(-P^n(J\# J))
\] 
is not slice.

For clarity and brevity, let $J^1=J^2=J$ and let $J^0=J\# J$.
Let $J_k^i=P^k(J^i)$ for $0\le k\le n$ and $0\le i\le 2$.
Note that $J^i_0 = P^0(J^i) = J^i$.
We have 
\[
K=J_n^1\# J_n^2\# (-J_n^0).
\] 
Suppose on the contrary that $K$ is slice.   We construct a 4-manifold $W_0$ as follows.
\begin{enumerate}
	\item Let $V$ be the exterior of a slice disk for $K$ in $D^4$. Then $\partial V=M(K)$ and $H_1(V)\cong \Z$. 
	\item Let $C$ be the standard cobordism $C(J_n^1, J_n^2, -J_n^0)$.
  See (2) in Section~\ref{subsection:proof-of-Theorem-A}.  Then, the top boundary of $C$ is
	\[
	\partial_+C = M(J_n^1)\sqcup M(J_n^2) \sqcup( -M(J_n^0))
	\]
  and the bottom boundary is $\partial_- C=-M(K)$. 
    
  \item For $0\le k\le n-1$ and $i=1,2$, let $E_k^i$ be the cobordism $E(P,J_k^i)$.
  See (3) in Section~\ref{subsection:proof-of-Theorem-A}.
  We have
  \[
  \partial E_k^i=M(J_k^i)\sqcup M(P(U)) \sqcup (-M(J_{k+1}^i)).
  \] 
  Similarly, let $E^0_k$ be the standard cobordism $-E(P,J_k^0)$. We have
  \[
  \partial E_k^0=(-M(J_k^0))\sqcup (-M(P(U))) \sqcup M(J_{k+1}^0).
  \]  
    For clarity, we denote the component $M(P(U))\subset \partial E_k^i$ by $M_k^i$, and the axis $\eta \subset M_k^i$ by  $\eta_k^i$ for $0\le k\le n-1$ and $0\le i\le 2$.

\end{enumerate}

Now, define $W_n=V\cup_{M(K)} C$. Then 
\[
\partial W_n=M(J_n^1)\sqcup M(J_n^2)\sqcup (-M(J_n^0)).
\]
For $k= n-1$, $n-2$, \ldots, $0$, define
\[
W_k=W_{k+1}\cupover{M(J_{k+1}^1)} E^1_k \cupover{M(J_{k+1}^2)} E^2_k\cupover{-M(J_{k+1}^0)} E^0_k.
\]
Then, we have 
\[
\partial W_k= \bigsqcup_{i=1}^2\Biggl(M(J_k^i)\sqcup\Biggl(\bigsqcup_{j=k}^{n-1}M_j^i\biggr)\Biggr)\sqcup 
\Biggl((-M(J_k^0))\sqcup\Biggl(\bigsqcup_{j=k}^{n-1}(-M_j^0)\Biggr)\Biggr).
\]
In particular, for $k=0$, we have
\begin{align*}
W_0 &=V\cup C\cup\Biggl(\bigsqcup_{i=0}^2E^i_{n-1}\Biggr)\cup \Biggl(\bigsqcup_{i=0}^2E^i_{n-2}\Biggr) \cup \cdots \cup \Biggl(\bigsqcup_{i=0}^2E^i_0\Biggr), \text{ and }
\\
\partial W_0 &= \bigsqcup_{i=1}^2\Biggl(M(J^i)\sqcup\Biggl(\bigsqcup_{j=0}^{n-1}M_j^i\Biggr)\Biggr)\sqcup 
\Biggl((-M(J^0))\sqcup\Biggl(\bigsqcup_{j=0}^{n-1}(-M_j^0)\Biggr)\Biggr).
\end{align*}
See Figure~\ref{figure:nonhomomorphism} for $W_0$.

\begin{figure}[ht]
  \includestandalone{Figure-nonhomomorphism}
  \caption{The cobordism $W_0$.}
  \label{figure:nonhomomorphism}
\end{figure}

Let $\phi\colon \pi_1(W_0)\to G$ be a homomorphism. By abuse of notation, we denote by $\phi$ restrictions of $\phi$ to subspaces of $W_0$. Recall that $S_G(W_0)=\sign_G^{(2)}(W_0)-\sign(W_0)$. By the definition of $\rhot$-invariants (see Section~\ref{section:preliminary}), we have
\begin{equation}
  \label{equation:section5-rho-invariant-of-W_0}
  \rhot(\partial W_0, \phi) = S_G(W_0).
\end{equation}
To derive a contradiction, we will compute each of $\rhot(\partial W_0, \phi)$ and $S_G(W_0)$ using different methods. 

We compute $S_G(W_0)$ following arguments in the proof of Theorem~\ref{theorem:main}. By Novikov additivity, 
\[
S_G(W_0)= S_G(V) + S_G(C) + \sum_{i=0}^2\sum_{k=0}^{n-1} S_G(E_k^i).
\]
We have $S_G(C)=0$ and $S_G(E_k^i)=0$ for all $i,k$, as discussed in (2) and (3) in Section~\ref{subsection:proof-of-Theorem-A}.

Suppose $\phi\colon \pi_1(W_0)\to G$ satisfies (H3) below. 
\begin{itemize}
	\item[(H3)] The codomain $G$ lies in Strebel's class $D(\Q)$, $G^{(n+1)}=\{e\}$, and $\phi$ sends a meridian of $K$ to an element of infinite order in $G$.
\end{itemize}
Then, it follows from the amenable signature theorem (Theorem~\ref{theorem:vanishing-rho-invariants}) that 
\[
S_G(V)=\rhot(M(K),\phi)=0
\] 
since the slice disk exterior $V$ is an integral $(n.5)$-solution for~$K$.
Therefore, $S_G(W_0)=0$.

Next, we compute $\rhot(\partial W_0,\phi)$. By additivity of $\rhot$-invariants, we have
\begin{align*}
\rhot(\partial W_0,\phi) &= \rhot(M(J^1),\phi) + \rhot(M(J^2),\phi) - \rhot(M(J^0),\phi) 
\\
& \qquad + \sum_{i=1}^2\sum_{k=0}^{n-1}\rhot(M_k^i,\phi) 
-\sum_{k=0}^{n-1}\rhot(M_k^0,\phi).
\end{align*}

Let $\phi_1$, $\phi_2$, and $\phi_0$ be the restrictions of $\phi$ to $\pi_1(M(J^1))$, $\pi_1(M(J^2))$, and $\pi_1(M(J^0))$, respectively.
Suppose that $\phi$ satisfies (H4) below.

\begin{itemize}
	\item[(H4)] The homomorphism $\phi_0$ is trivial.
  Each of $\phi_1$ and $\phi_2$ is trivial or has image isomorphic to~$\Z$.
  Either $\phi_1$ or $\phi_2$ has image isomorphic to~$\Z$.
\end{itemize}
By Lemma~\ref{lemma:properties-of-rho-invariants} and by our choice of $J$, it follows that 
\[
  \rhot(M(J^1),\phi) + \rhot(M(J^2),\phi) > 3nL \quad\text{and}\quad \rhot(M(J^0),\phi)=0.
\]
Therefore, by our choice of $L$
\[
\rhot(\partial W_0,\phi) > 3nL -0 - 2nL - nL =0.
\]
Since $S_G(W_0) = 0$ by the above, it contradicts~\eqref{equation:section5-rho-invariant-of-W_0}.
Therefore, the proof is completed by the proposition below.

\begin{proposition}\label{proposition:computation-of-rho(J)-non-homomorphism}
There exists a homomorphism $\phi\colon \pi_1(W_0)\to G$ that satisfies both (H3) and (H4).
\end{proposition}
We give a proof of Proposition~\ref{proposition:computation-of-rho(J)-non-homomorphism} in Section~\ref{subsection:proof-of-non-homomorphism}.

\subsection{Proof of Proposition~\ref{proposition:computation-of-rho(J)-non-homomorphism}}\label{subsection:proof-of-non-homomorphism}
In this subsection, we prove Proposition~\ref{proposition:computation-of-rho(J)-non-homomorphism}. 

\begin{proof}[Proof of Proposition~\ref{proposition:computation-of-rho(J)-non-homomorphism}]
Suppose that $X$ is a space containing $V$ as a subspace such that the inclusion-induced map $H_1(V)\to H_1(X)$ is an isomorphism.
The main example is $X=W_k$.
For $\Gamma=\pi_1(X)$, we define its subgroups $\cP^k\Gamma$ for $k=0, 1,\dots, n+1$ as follows. 
Let $\cP^0\Gamma=\Gamma$ and let
\[
\cP^1\Gamma = \Gamma^{(1)}=[\Gamma,\Gamma] = \Ker\left\{\Gamma\rightarrow H_1(\Gamma)=\Z=\langle t\rangle\right\}.
\]

Recall that we use the cobordism $C$ bounded by $\partial_+C \sqcup - \partial_-C$ in the construction of~$W_0$.
Let $i_+\colon \partial_+ C\to C$ and $i_-\colon \partial_- C\to C$ be inclusions. Then, the composition of the inclusion-induced homomorphisms 
\[
(i_-)_*^{-1}\circ (i_+)_*\colon H_1(\partial_+C;\Q[t^{\pm 1}])\to H_1(\partial_-C;\Q[t^{\pm 1}])
\] 
is an isomorphism. Via this isomorphism, we obtain an identification
\begin{equation}\label{equation:identification}
  \begin{aligned}
    H_1(M(K);\Q[t^{\pm 1}])
    &= H_1(M(J_n^1);\Q[t^{\pm 1}])\oplus H_1(M(J_n^2);\Q[t^{\pm 1}])
    \\
    & \qquad \oplus H_1(-M(J_n^0);\Q[t^{\pm 1}]),
  \end{aligned}
\end{equation}
which we use throughout this subsection.
Note that the axis $\eta_k^i$ is identified with a meridian $\mu_k^i$ of the knot $J_k^i$ in $M(J_{k+1}^i)$. 
Let $A$ be the $\Q[t^{\pm1}]$-submodule of $H_1(-M(J_n^0);\Q[t^{\pm 1}])$ generated by $\eta_{n-1}^0$. Let 
\[
i_*^X\colon A\to H_1(M(K);\Q[t^{\pm1}]) \to H_1(V;\Q[t^{\pm 1}])\to H_1(X;\Q[t^{\pm 1}])
\] 
be the homomorphism induced by the inclusions. 

Define $\cP^2\Gamma$ to be the kernel of the following composition:
\[
\cP^1\Gamma\to \frac{\cP^1\Gamma}{[\cP^1\Gamma,\cP^1\Gamma]}\xrightarrow{\cong} H_1(X;\Z[t^{\pm 1}])\to
H_1(X;\Q[t^{\pm 1}]) \xrightarrow{q} H_1(X;\Q[t^{\pm 1}])/i_*^X(A).
\]
For $k\ge 2$, define 
\[
\cP^{k+1}\Gamma = \Ker\left\{\cP^k\Gamma \to \frac{\cP^k\Gamma}{[\cP^k\Gamma,\cP^k\Gamma]}\to \frac{\cP^k\Gamma}{[\cP^k\Gamma,\cP^k\Gamma]}\otimes_\Z \Q = H_1\bigl(\Gamma;\Q[\Gamma/\cP^k\Gamma]\bigr)\right\}.
\]
Finally, let $G=\pi_1(W_0)/\cP^{n+1}\pi_1(W_0)$ and define $\phi\colon \pi_1(W_0)\to G$ to be the quotient homomorphism.
We will show that $\phi$ satisfies (H3) and~(H4).
For the reader's convenience we recall (H3) and (H4) below.

\begin{itemize}
	\item[(H3)] The codomain $G$ lies in Strebel's class $D(\Q)$, $G^{(n+1)}=\{e\}$, and $\phi$ sends a meridian of $K$ to an element of infinite order in $G$.
	\item[(H4)] The homomorphism $\phi_0$ is trivial.
  Each of $\phi_1$ and $\phi_2$ is trivial or has image isomorphic to~$\Z$.
  Either $\phi_1$ or $\phi_2$ has image isomorphic to~$\Z$.
\end{itemize}

\begin{proof}[Proof of (H3)]
Since $\pi_1(W_0)^{(n+1)}\subset \cP^{n+1}\pi_1(W_0)$, it follows that $G^{(n+1)}=\{e\}$.
For $0\le k\le n$, the quotient $\cP^k\pi_1(W_0)/\cP^{k+1}\pi_1(W_0)$ is torsion-free abelian since it injects into a torsion-free abelian group.
It follows from Lemma~\ref{lemma:amenable-D(R)-group}~(1) that the group $G$ lies in Strebel's class~$D(\Q)$. A meridian of $K$ generates
  \[
    \pi_1(W_0)/\cP^1  \pi_1(W_0) \cong H_1(W_0) \cong H_1(M(K)) = \Z    
 \]
onto which $G$ surjects. Therefore, $\phi$ sends a meridian of $K$ to an infinite order element in~$G$.  This proves $(H3)$.
\let\qed=\relax 
\end{proof}

\begin{proof}[Proof of (H4)]
Recall that $\mu_k^i$ is a meridian of $J_k^i$. Since $\mu_0^i$, $0\le i\le 2$, generates $H_1(M(J^i))\cong \Z$ and the group $\cP^n\pi_1(W_0)/\cP^{n+1}\pi_1(W_0)$ is torsion-free abelian, $\phi$ satisfies (H4) if it satisfies (H4$'$) below.
\begin{itemize}
	\item[(H4$'$)] The homomorphism $\phi\colon \pi_1(W_0)\to G$ sends $\mu_0^i$ to $\cP^n\pi_1(W_0)/\cP^{n+1}\pi_1(W_0)$ for each~$i$, sends $\mu_0^0$ to the identity, and sends either $\mu_0^1$ or $\mu_0^2$ to a nontrivial element.
  \end{itemize}

Using arguments in the proof of Proposition~\ref{proposition:computation-of-rho-invariant-of-J_0^i}, one can show that $\mu_k^i\in \cP^{n-k}\pi_1(W_0)$ for each $i$ and $k$. In particular, $\mu_0^i\in \cP^n\pi_1(W_0)$ for each~$i$.
From this the first part of (H4$'$) follows.

We will show that $\mu_0^0$ is trivial in $\cP^n\pi_1(W_0)/\cP^{n+1}\pi_1(W_0)$ using an induction argument. By our choice of the submodule $A$, it follows that $\eta_{n-1}^0$ is trivial in $H_1(W_n;\Q[t^{\pm 1}])/i_*^{W_n}(A)$ and hence lies in $\cP^2\pi_1(W_n)$. From arguments in \cite[Assertion~1 on p.~4799]{Cha:2010-01}, it follows that for $0\le k\le n-1$,
\[
\pi_1(W_k)/\pi_1(W_k)^{(2)}\cong  \pi_1(W_n)/\pi_1(W_n)^{(2)}.
\]
In particular $H_1(W_k;\Q[t^{\pm 1}])\cong H_1(W_n;\Q[t^{\pm 1}])$ and
\begin{equation}\label{equation:isomorphism-metabelian}
H_1(W_k;\Q[t^{\pm 1}])/i_*^{W_k}(A)\cong H_1(W_n;\Q[t^{\pm 1}])/i_*^{W_n}(A).
\end{equation}
Therefore,
\begin{equation}\label{equation:isomorphism-cP1-cP2}
\cP^1\pi_1(W_k)/\cP^2\pi_1(W_k) \cong \cP^1\pi_1(W_n)/\cP^2\pi_1(W_n).
\end{equation}
Therefore, $\cP^2\pi_1(W_{n-1})\cong \cP^2\pi_1(W_n)$, and it follows that $\eta_{n-1}^0$ lies in $\cP^2\pi_1(W_{n-1})$. Since $\eta_{n-1}^0$ is identified with $\mu_{n-1}^0$, we have $\mu_{n-1}^0\in \cP^2\pi_1(W_{n-1})$. 

Now suppose that $\mu_{k+1}^0\in \cP^{n-k}\pi_1(W_{k+1})$ where $k\le n-2$. We will show that $\mu_k^0\in \cP^{n-k+1}\pi_1(W_k)$. Note that $\mu_k^0$ is identified with $\eta_k^0$ in  in $M(J_{k+1}^0)\subset W_k$.
Since the satellite operator $P$ has axis $\eta_k^0$ and has winding number zero, it follows that $\mu_k^0\in \langle \mu_{k+1}^0\rangle^{(1)}$ in $\cP^{n-k}\pi_1(W_k)$. 
Therefore,
\[
\mu_k^0\in \bigl(\cP^{n-k}\pi_1(W_k)\bigr)^{(1)} \subset \cP^{n-k+1}\pi_1(W_k).
\]
Inductively, we can conclude that $\mu_0^0\in \cP^{n+1}\pi_1(W_0)$, and it follows that $\mu_0^0$ is trivial in $\cP^n\pi_1(W_0)/\cP^{n+1}\pi_1(W_0)$.

We show that either $\mu_0^1$ or $\mu_0^2$ is nontrivial in $\cP^n\pi_1(W_0)/\cP^{n+1}\pi_1(W_0)$ using an induction argument as in the proof of Proposition~\ref{proposition:computation-of-rho-invariant-of-J_0^i}.

First, we show that either $\mu_{n-1}^1$ or $\mu_{n-1}^2$ is nontrivial in $\cP^1\pi_1(W_{n-1})/\cP^{2}\pi_1(W_{n-1})$. 
Recall that $\mu_{n-1}^i$ is identified with $\eta_{n-1}^i$ in $M(J_n^i)$, hence in $W_n$. Therefore, by (\ref{equation:isomorphism-cP1-cP2}) we only need to show that either $\eta_{n-1}^1$ or $\eta_{n-1}^2$ is nontrivial in $\cP^1\pi_1(W_n)/\cP^{2}\pi_1(W_n)$. Since $\cP^1\pi_1(W_n)/\cP^2\pi_1(W_n)$ injects into $H_1(W_n;\Q[t^{\pm 1}])/i_*^{W_n}(A)$, it suffices to show that either $\eta_{n-1}^1$ or $\eta_{n-1}^2$ is nontrivial in $H_1(W_n;\Q[t^{\pm 1}])/i_*^{W_n}(A)$.

Using an argument similar to the one for \eqref{equation:isomorphism-metabelian}, one can see that 
\[
H_1(W_n;\Q[t^{\pm 1}])/i_*^{W_n}(A)\cong H_1(V;\Q[t^{\pm 1}])/i_*^V(A).
\]
Therefore, we only need to show that either $\eta_{n-1}^1$ or $\eta_{n-1}^2$ is nontrivial in $H_1(V;\Q[t^{\pm 1}])/i_*^V(A)$.

Let 
\[
M=\Ker\{i_*\colon H_1(M(K);\Q[t^{\pm 1}])\to H_1(V;\Q[t^{\pm 1}])\}.
\]
Then the map 
\[
q\circ i_*\colon H_1(M(K);\Q[t^{\pm 1}])\to H_1(V;\Q[t^{\pm 1}])\to H_1(V;\Q[t^{\pm1 }])/i_*^V(A)
\] 
induces the homomorphism
\[
\tilde{i_*}\colon H_1(M(K);\Q[t^{\pm 1}])/(M+A)\to H_1(V;\Q[t^{\pm 1}])/i_*^V(A).
\]
Since $M$ is the kernel of the map $i_*$, one can readily see that $\tilde{i_*}$ is injective. 
Therefore, if $\eta_{n-1}^i$ is nontrivial in $H_1(M(K);\Q[t^{\pm 1}])/(M+A)$, then it will be nontrivial in $H_1(V;\Q[t^{\pm 1}])/i_*^V(A)$.

For brevity, let $\eta^i=\eta_{n-1}^i$, $0\le i\le 2$. Suppose on the contrary that both $\eta^1$ and $\eta^2$ are trivial in $H_1(M(K);\Q[t^{\pm 1}])/(M+A)$. 
Let $\Bl_K$ denote the Blanchfield form on $H_1(M(K);\Q[t^{\pm 1}])$ and let $\Bl^i$ denote the Blanchfield form on $H_1(M(J_n^i);\Q[t^{\pm 1}])$ for each $i$. Then, via the identification \eqref{equation:identification}, we have
\[
\Bl\nolimits_K = B\ell^1 \oplus B\ell^2\oplus (-B\ell^0).
\]
Since the satellite operator $P$ has winding number zero, it follows that 
\[
H_1(M(J_n^i);\Q[t^{\pm 1}])\cong H_1(M(P(U));\Q[t^{\pm 1}])
\] 
for all $i$ and all $\Bl^i$ are isomorphic to the Blanchfield form $\Bl$ on $H_1(M(P(U));\Q[t^{\pm 1}])$. Let $g(t)$ and $h(t)$ be relatively prime polynomials in $\Q[t^{\pm1}]$ such that $\Bl(\eta,\eta)=h(t)/g(t)$. Since $\Bl(\eta,\eta)$ is nontrivial by our assumption, we have $\deg (g(t))>0$. Note that $\Bl^i(\eta^i,\eta^i)=h(t)/g(t)$ for all $i$.

Since $\eta^1$ and $\eta^2$ are trivial in $H_1(M(K);\Q[t^{\pm 1}])/(M+A)$ and the submodule $A$ is generated by $\eta^0$, there exists $f_i(t)\in \Q[t^{\pm1}]$ for each $i=1,2$ such that $\eta^i - f_i(t)\eta^0\in M$.

Since $V$ is the exterior of a slice disk for $K$ and
\[
M=\Ker\{i_*\colon H_1(M(K);\Q[t^{\pm 1}])\to H_1(V;\Q[t^{\pm 1}])\},
\]
it follows that $M$ is a metabolizer for the Blanchfield form $\Bl_K$ on $H_1(M(K);\Q[t^{\pm 1}])$. In particular, $\Bl_K(x,y)=0$ for all $x,y\in M$.

Then, it follows that in $\Q(t)/\Q[t^{\pm1}]$, for each $i=1,2$, 
\begin{align*}
0	&=\Bl\nolimits_K\left(\eta^i - f_i(t)\eta^0,\eta^i - f_i(t)\eta^0\right) \\
	&= B\ell^i\left(\eta^i,\eta^i\right) - f_i(t)f_i(t^{-1})B\ell^0\left(\eta^0,\eta^0\right) \\
	&= \left(1-f_i(t)f_i(t^{-1})\right)h(t)/g(t)
\end{align*}
It follows that $g(t)$ divides $1-f_i(t)f_i(t^{-1})$ in $\Q[t^{\pm1}]$ for $i=1,2$. Also, since $M$ is a $\Q[t^{\pm 1}]$-submodule of $H_1(M(K);\Q[t^{\pm 1}])$, we have
\[
f_2(t)\eta^1 - f_1(t)\eta^2 = f_2(t)\left(\eta^1 - f_1(t)\eta^0\right) - f_1(t)\left(\eta^2 - f_2(t)\eta^0\right) \in M.
\]
Therefore,
\begin{align*}
0	&=\Bl\nolimits_K\left(f_2(t)\eta^1 - f_1(t)\eta^2,f_2(t)\eta^1 - f_1(t)\eta^2\right) \\
	&= f_2(t)f_2(t^{-1}) B\ell^1\left(\eta^1,\eta^1\right) + f_1(t)f_1(t^{-1})B\ell^2\left(\eta^2,\eta^2\right) \\
	&= \left(f_1(t)f_1(t^{-1}) + f_2(t)f_2(t^{-1})\right)h(t)/g(t),
\end{align*}
and it follows that $g(t)$ divides $f_1(t)f_1(t^{-1})+f_2(t)f_2(t^{-1})$ in $\Q[t^{\pm1}]$. 

Now $g(t)$ divides $1-f_1(t)f_1(t^{-1})$, $1-f_2(t)f_2(t^{-1})$, and $f_1(t)f_1(t^{-1})+f_2(t)f_2(t^{-1})$ in $\Q[t^{\pm1}]$. Therefore, $g(t)$ divides 2 in $\Q[t^{\pm1}]$, which contradicts that $\deg (g(t))>0$. Therefore, either $\eta^1$ or $\eta^2$ is nontrivial in $H_1(M(K);\Q[t^{\pm 1}])/(M+A)$. Consequently, this shows that either $\mu_{n-1}^1$ or $\mu_{n-1}^2$ is nontrivial in $\cP^1\pi_1(W_{n-1})/\cP^{2}\pi_1(W_{n-1})$.

Exchanging $\mu_{n-1}^1$ with $\mu_{n-1}^2$ if necessary, we may assume that $\mu_{n-1}^1$ is nontrivial in the quotient $\cP^1\pi_1(W_{n-1})/\cP^{2}\pi_1(W_{n-1})$. 
Now we use an induction for $k$ in the reverse order as in the proof of Proposition~\ref{proposition:computation-of-rho-invariant-of-J_0^i}. 
Suppose that $\mu_{k+1}^1$ is nontrivial in $\cP^{n-k-1}\pi_1(W_{k+1})/\cP^{n-k}\pi_1(W_{k+1})$ where $k\le n-2$. 
Then, using an argument in the proof of Proposition~\ref{proposition:computation-of-rho-invariant-of-J_0^i}, one can show that $\mu_k^1$ is nontrivial in $\cP^{n-k}\pi_1(W_k)/\cP^{n-k+1}\pi_1(W_k)$. In doing so, we note that for $k\le n-1$, the pair $(W_{k+1},\pi\to \Gamma)$ with $\pi=\pi_1(W_{k+1})$ and $\Gamma=\pi/\cP^{n-k}\pi$ is a $\Q$-coefficient Blanchfield bordism by Lemmas~\ref{lemma:lagrangian-bordism-extension}, \ref{lemma:C-and-E-injectivity} and \ref{lemma:integral-solution-lagrangian-bordism} and Proposition~\ref{proposition:lagrangian-Blanchfield-bordism}.

Now by an induction, $\mu_0^1$ is nontrivial in $\cP^n\pi_1(W_0)/\cP^{n+1}\pi_1(W_0)$, and $\phi$ satisfies (H4$'$). 
This completes the proof of Proposition~\ref{proposition:computation-of-rho(J)-non-homomorphism}.
\end{proof}
\let\qed=\relax
\end{proof}

\begin{remark}\label{remark:non-homomorphism-generalization}
Theorem~\ref{theorem:non-homomorphism} generalizes in various ways.

\begin{enumerate} 
	\item The satellite operators in Theorem~\ref{theorem:non-homomorphism} are not homomorphisms even modulo $\cF_{n+1}$: in the proof of Theorem~\ref{theorem:non-homomorphism}, let $V$ be an integral $(n+1)$-solution for $M(K)$. The group $G$ in Proposition~\ref{proposition:computation-of-rho(J)-non-homomorphism} defined in Section~\ref{subsection:proof-of-non-homomorphism} satisfies $G^{(n+1)}=\{e\}$, and therefore we still have $S_G(V)=0$ by the amenable signature theorem (Theorem~\ref{theorem:vanishing-rho-invariants}).
	\item Theorem~\ref{theorem:non-homomorphism} generalizes to the case of compositions of distinct satellite operators. See Theorem~\ref{theorem:non-homomorphism-generalization} below, whose proof is essentially identical to that of Theorem~\ref{theorem:non-homomorphism}\@. Theorem~\ref{theorem:non-homomorphism-generalization} generalizes \cite[Corollary~4.5]{Miller:2019-01} by removing the condition in \cite[Corollary~4.5]{Miller:2019-01} on the existence of a knot $K$ such that $\rho_0(K)\ne 0$ and $P(K)$ is slice; we note that to the authors' knowledge it is unknown whether or not there exists such a knot $K$ for each winding number zero satellite operator. 
\end{enumerate}
\end{remark}

\begin{theorem}[{Generalization of Theorem~\ref{theorem:non-homomorphism}}]\label{theorem:non-homomorphism-generalization}
For $1\le k\le n$, let $P_k\colon \cC\to \cC$ be a winding number zero satellite operator with axis $\eta_k$. For each $k$, let $\Bl_k$ be the Blanchfield form of~$P_k(U)$. If $\Bl_k(\eta_k, \eta_k)$ is nontrivial for $1\le k\le n$, then $P_n\circ P_{n-1}
\circ \cdots \circ P_1\colon \cC \to \cC$ is not a homomorphism, even modulo $\cF_{n+1}$.
\end{theorem}
\begin{proof}
The proof is almost the same as that of Theorem~\ref{theorem:non-homomorphism}.
So we only give a list of necessary changes.
\begin{enumerate}
	\item Let $L$ be a constant such that $|\rhot(M(P_k(U)),\phi)|< L$ for all $\phi$ and $1\le k\le n$.
	\item Let $J_k^i=P_k(J_{k-1}^i)$, $1\le k\le n$.
	\item For $0\le k\le n-1$, let $E_k^i=E(P_{k+1}, J_k^i)$ for $i=1,2 $ and let $E_k^0=-E(P_{k+1}, J_k^0)$.
  \qedhere
\end{enumerate}
\end{proof}

\section{Satellites of patterns with relatively prime Alexander polynomials}\label{section:coprime-Alexander-polynomial}

In this section we prove Theorem~\ref{theorem:coprime-satellites}:
suppose that $P\subset S^1\times D^2$ has winding number zero and $\Bl(\eta, \eta)$ is nontrivial for the axis $\eta$ of~$P$.
Suppose that $Q\subset S^1\times D^2$ has winding number zero and the Alexander polynomials of $P(U)$ and $Q(U)$ are relatively prime.
Our goal is to prove that the operators $P^n\colon \cC \to \cC$ and $Q\colon \cC \to \cC$ are distinct for all $n\ge 0$.

\subsection{Proof of Theorem~\ref{theorem:coprime-satellites}}\label{subsection:proof-theorem-coprime}

By Theorem~\ref{theorem:universal-bound}, there exist $L_1, L_2>0$ such that 
\begin{equation}
  \label{equation:bounds-L_1-L_2}
  |\rhot(M(P(U)),\phi)| < L_1 \text{ and } |\rhot(M(-P^n(U), \psi)| < L_2
\end{equation}
for all group homomorphisms $\phi\colon \pi_1(M(P(U)))\to G$ and $\psi\colon \pi_1(M(-P^n(U)))\to G$. 

Let $J$ be a knot such that $\rhot(M(J),\epsilon) > nL_1 + L_2$ where $\epsilon\colon \pi_1(M(J))\to \Z$ is the abelianization. For instance, as~$J$, one can use a connected sum of sufficiently many copies of the left-handed trefoil.

If $P^n(U)\ne Q(U)$ in $\cC$, then we are done.
Suppose that $P^n(U)=Q(U)$ in~$\cC$. We will show that $P^n(J)\ne Q(J)$ in~$\cC$, to conclude that $P^n \ne Q$ on~$\cC$.

Suppose on the contrary that $P^n(J)=Q(J)$ in $\cC$.  Let 
\[
K=P^n(J)\#(-P^n(U))\#  (-Q(J)) \# Q(U).
\]
Then $K$ is slice since $P^n(U)$ and $Q(U)$ are concordant.

We construct a 4-manifold $W_0$ as follows. 

\begin{enumerate}
	\item Let $V$ be the exterior of a slice disk for $K$ in $D^4$. Note that $\partial V=M(K)$. 
	\item Let $C = C\bigl(P^n(J),-P^n(U),-Q(J)\# Q(U)\bigr)$ be the standard cobordism. See (2) in Section~\ref{subsection:proof-of-Theorem-A}.
  The top boundary of $C$ is
	\[
	\partial_+C = M(P^n(J))\sqcup M(-P^n(U)) \sqcup M((-Q(J))\# Q(U))
	\]
	and the bottom boundary is $\partial_- C=-M(K)$. 
	\item Recall that $P^{k+1}(J)=P(P^k(J))$, where $P^0(J)=J$. For $0\le k\le n-1$, let $E_k = E(P,P^k(J))$ be the cobordism with boundary 
  \[
    \partial E_k=M(P^k(J))\sqcup M(P(U)) \sqcup (-M(P^{k+1}(J))).
  \]
  that we used in (3) in Section~\ref{subsection:proof-of-Theorem-A}.
	\item $P^n(J)\#(-P^n(U))$ is integrally $n$-solvable by Lemma~\ref{lemma:affine-inclusion}. Since $K$ is slice and $\cF_n$ is a subgroup of $\cC$, it follows that $(-Q(J))\# Q(U)$ is integrally $n$-solvable. Let $Z$ be an integral $n$-solution for $(-Q(J))\# Q(U)$.
\end{enumerate}

Let
\[
W_n=V\cupover{M(K)} C \cupover{M((-Q(J))\#Q(U))} Z.
\]
Then $\partial W_n=M(P^n(J))\sqcup M(-P^n(U))$. 

For $k = n-1$, $n$,~\ldots,~$0$, define inductively
\[
W_k=W_{k+1}\cupover{M(P^{k+1}(J))} E_k.
\]
Then we have 
\begin{align*}
W_0&=V\cup C \cup Z \cup E_{n-1}\cup E_{n-2}\cup \cdots \cup E_0, \text{ and }\\
\partial W_0&=M(J)\sqcup nM(P(U))\sqcup M(-P^n(U)).
\end{align*}
See Figure~\ref{figure:coprime} for the cobordism $W_0$.

\begin{figure}[ht]
  \includestandalone{Figure-coprime}
  \caption{The cobordism $W_0$}
  \label{figure:coprime}
\end{figure}

Let $\phi\colon \pi_1(W_0)\to G$ be a homomorphism. By abuse of notation we write $\phi$ for restrictions of $\phi$ to subspaces of $W_0$ as well. 

Recall that for a codimension zero submanifold $X$ of $W_0$, we denote the $L^2$-signature defect by $S_G(X) = \sign_G^{(2)}(X)-\sign(X)$.
We have
\begin{equation}
  \label{equation:section6-rho-invariant-of-W_0}
  \rhot(\partial W_0, \phi) = S_G(W_0).
\end{equation}
For a specific choice of $\phi$ that we will make later, we will show that $\rhot(\partial W_0,\phi)>0$ and $S_G(W_0)=0$.
This will finish the proof.

First, we compute $S_G(W_0)$. By Novikov additivity, 
\[
S_G(W_0) = S_G(V) + S_G(C) + \sum_{k=0}^{n-1}S_G(E_k) + S_G(Z). 
\]

Using the same arguments given in the proof of Theorem~\ref{theorem:main}, one can show that $S_G(C)=0$ and $S_G(E_k)=0$ for all $k$. Suppose that $\phi\colon \pi_1(W_0)\to G$ satisfies (H5) below.

\begin{itemize}
	\item[(H5)] The codomain $G$ lies in Strebel's class $D(\Q)$ and $G^{(n+1)}=\{e\}$. The image of $\pi_1(Z)^{(n)}\to \pi_1(W_0)$ lies in the kernel of $\phi$, and $\phi$ sends a meridian of $(-Q(J))\# Q(U)$ in $\partial Z=M((-Q(J))\# Q(U))$ to an element of infinite order in~$G$.
\end{itemize}
Then, since the slice disk exterior $V$ is an integral $(n.5)$-solution and $G^{(n+1)}=\{e\}$, it follows that $S_G(V)$=0 by the amenable signature theorem (Theorem~\ref{theorem:vanishing-rho-invariants}). 
Also,  $\phi\colon \pi_1(Z)\to G$ factors through the quotient map $\pi_1(Z)\to \pi_1(Z)/\pi_1(Z)^{(n)}$. Let $\Gamma$ be the image of $\pi_1(Z)/\pi_1(Z)^{(n)}$ in $G$. Let $\psi\colon \pi_1(Z)\to \Gamma$ be the map induced from $\phi$. By Lemma~\ref{lemma:properties-of-rho-invariants}~(1), $\rhot(\partial Z,\phi)=\rhot(\partial Z, \psi)$. Note that $\Gamma^{(n)}=\{e\}$. Since an integral $n$-solution is an integral $(n-0.5)$-solution,  $Z$ is an integral $(n-0.5)$-solution.
By Theorem~\ref{theorem:vanishing-rho-invariants} (with $n$ in place of $n+1$), $\rhot(\partial Z, \psi)=0$, and hence $\rhot(\partial Z, \phi)=0$. Since $\phi$ extends over $Z$, it follows that $S_G(Z)=0$. Therefore, $S_G(W_0)=0$. 

We compute $\rhot(\partial W_0, \phi)$.
Recall that $\partial W_0$ contains $n$ copies of~$M(P(U))$.
For clarity, denote the restrictions of $\phi$ to the $i$th copy of $M(P(U))$ in $\partial W_0$ by $\phi_i$, $1\le i\le n$. By additivity of $\rhot$-invariants, it follows that 
\begin{equation}
  \label{equation:rhot-partial_W_0}
  \rhot(\partial W_0, \phi) = \rhot(M(J),\phi) + \sum_{i=1}^n\rhot(M(P(U)),\phi_i) +\rhot(M(-P^n(U)),\phi).
\end{equation}
Suppose that $\phi$ satisfies (H6) below.
\begin{itemize}
	\item[(H6)] The image of $\phi\colon \pi_1(M(J))\to G$ is isomorphic to $\Z$.
\end{itemize}
Then, by Lemma~\ref{lemma:properties-of-rho-invariants}~(1), $\rhot(M(P(J)),\phi) = \rhot(M(J),\epsilon)$ where $\epsilon\colon \pi_1(M(J))\to~\Z$ is the abelianization. Therefore, by~\eqref{equation:rhot-partial_W_0} and~\eqref{equation:bounds-L_1-L_2} and by the choice of $J$, we have
\[
  \rhot(\partial W_0, \phi) > \rhot(M(J),\epsilon) - nL_1 - L_2 > 0.
\]

The above results $S_G(W_0)=0$ and $\rhot(\partial W_0,\phi)>0$ contradict~\eqref{equation:section6-rho-invariant-of-W_0}.
Therefore, the proposition below completes the proof of Theorem~\ref{theorem:coprime-satellites}\@.

\begin{proposition}\label{proposition:phi-coprime}
There exists a homomorphism $\phi\colon \pi_1(W_0)\to G$ satisfying (H5) and~(H6).
\end{proposition}

\subsection{Proof of Proposition~\ref{proposition:phi-coprime}}\label{subsection:proof-proposition-nonhomomorphism}

In this subsection, we give a proof of Proposition~~\ref{proposition:phi-coprime}.
For the reader's convenience, we recall (H5) and (H6) below.

\begin{itemize}
	\item[(H5)] The codomain $G$ lies in Strebel's class $D(\Q)$ and $G^{(n+1)}=\{e\}$. The image of $\pi_1(Z)^{(n)}\to \pi_1(W_0)$ lies in the kernel of $\phi$, and $\phi$ sends a meridian of $(-Q(J))\# Q(U)$ in $\partial Z=M((-Q(J))\# Q(U))$ to an element of infinite order in~$G$.
	\item[(H6)] The image of $\phi\colon \pi_1(M(J))\to G$ is isomorphic to $\Z$.
\end{itemize}

\begin{proof}[Proof of Proposition~\ref{proposition:phi-coprime}]
Let $X$ be a subspace of $W_0$ containing $W_n$ as a subspace such that the inclusion-induced map $H_1(W_n)\to H_1(X)$ is an isomorphism.
For example $X=W_k$ for any $k\le n$.

Let $\Gamma=\pi_1(X)$. For $k\ge 0$, we define $\cP^k\Gamma$ as follows. 
Let $\cP^0\Gamma=\Gamma$. Let
\[
\cP^1\Gamma = \Gamma^{(1)}=[\Gamma,\Gamma] = \Ker\left\{\Gamma\rightarrow H_1(\Gamma)=\Z=\langle t\rangle\right\}.
\] 
We define $\cP^2\Gamma$ following \cite[Section~3.1]{Cha:2021-1} as detailed below.
Let $\lambda(t)$ be the Alexander polynomial of $P(U)$. Let 
\[
\Sigma = \{f(t)\in \Q[t^{\pm 1}]\mid f(1)\ne 0,\; \gcd(f(t), \lambda(t))=1\},
\]
which is a multiplicatively closed subset of $\Q[t^{\pm 1}]$.
It follows that the localization $\Q[t^{\pm 1}]\Sigma^{-1}$ is defined. Let $\Im H_1(Z;\Q[t^{\pm 1}]\Sigma^{-1})$ be the image of the inclusion-induced map 
\[
H_1(Z;\Q[t^{\pm 1}]\Sigma^{-1})\to H_1(X;\Q[t^{\pm 1}]\Sigma^{-1}).
\]
Now we define $\cP^2\Gamma$ to be the kernel of the following map:
\begin{align*}
\cP^1\Gamma &\to \frac{\cP^1\Gamma}{[\cP^1\Gamma,\cP^1\Gamma]}\to \frac{\cP^1\Gamma}{[\cP^1\Gamma,\cP^1\Gamma]}\otimes_\Z \Q = H_1\big(\Gamma;\Q\left[\Gamma/\cP^1\Gamma\right]\big) = H_1\big(X;\Q[t^{\pm 1}]\big) \\
&\quad \to H_1\big(X;\Q[t^{\pm 1}]\Sigma^{-1}\big) \to H_1\big(X;\Q[t^{\pm 1}]\Sigma^{-1}\big)/\Im H_1\big(Z;\Q[t^{\pm 1}]\Sigma^{-1}\big).
\end{align*}
For $k\ge 2$, define
\[
\cP^{k+1}\Gamma = \Ker\left\{\cP^k\Gamma \to \frac{\cP^k\Gamma}{[\cP^k\Gamma,\cP^k\Gamma]}\to \frac{\cP^k\Gamma}{[\cP^k\Gamma,\cP^k\Gamma]}\otimes_\Z \Q = H_1\bigl(\Gamma;\Q[\Gamma/\cP^k\Gamma]\bigr)\right\}.
\]

Now let $G=\pi_1(W_0)/\cP^{n+1}\pi_1(W_0)$, and define $\phi\colon \pi_1(W_0)\to G$ to be the quotient map. We will show that the $\phi$ satisfies (H5) and~(H6).

\begin{proof}[Proof of (H5)]
We have $G^{(n+1)}=\{e\}$ since $\pi_1(W_0)^{(n+1)}\subset \cP^{n+1}\pi_1(W_0)$.
Since the quotient $\cP^k\pi_1(W_0)/\cP^{k+1}\pi_1(W_0)$ injects into a torsion-free abelian group, it is also torsion-free abelian. Therefore, it follows from Lemma~\ref{lemma:amenable-D(R)-group}~(1) that $G$ lies in Strebel's class $D(\Q)$.

By our choice of $\cP^2\Gamma$, it follows from \cite[p.~11, Assertion~B]{Cha:2021-1} that the map induced by the inclusion $Z\hookrightarrow W_0$ sends $\pi_1(Z)^{(n)}$ into $\cP^{n+1}\pi_1(W_0)$. Therefore, $\phi$ sends $\pi_1(Z)^{(n)}$ to $\{e\} \subset G$. 

The homomorphism $\phi$ sends a meridian of $\partial Z$ to an element of infinite order in $G$ since $\cP^0\pi_1(W_0)/\cP^1\pi_1(W_0) \cong H_1(W_0)\cong H_1(\partial Z)$, which is generated by a meridian of $\partial Z$. Therefore, $\phi$ satisfies~(H5).
\def\qedsymbol{}
\end{proof}

\begin{proof}[Proof of (H6)]
It suffices to show that $\phi$ sends a meridian of $J$ to a nontrivial element in $\cP^n\pi_1(W_0)/\cP^{n+1}\pi_1(W_0)$ since $\cP^n\pi_1(W_0)/\cP^{n+1}\pi_1(W_0)\subset G$ is torsion-free abelian.  To show this, we follow the argument in the proof of Proposition~\ref{proposition:computation-of-rho-invariant-of-J_0^i}.

Let $J_0=J$ and $J_k=P^k(J_0)$ for $1\le k\le n$. Let $\mu_k$ be a meridian of $J_k$. 

For $k=n$, the meridian $\mu_n$ is freely homotopic to the meridian of $\partial Z$, which is nontrivial in $\cP^0\pi_1(W_n)/\cP^1\pi_1(W_n)\cong H_1(W_n)\cong H_1(\partial Z)$ as we saw in the proof of~(H5). 

For $k=n-1$, we assert that $\mu_{n-1}$ is nontrivial in $\cP^1\pi_1(W_{n-1})/\cP^2\pi_1(W_{n-1})$. 
By our choice of $\cP^2\pi_1(W_{n-1})$, we only need to show that $\mu_{n-1}$ is nontrivial in the quotient $H_1(W_{n-1};\Q[t^{\pm 1}]\Sigma^{-1})/\Im H_1(Z;\Q[t^{\pm 1}]\Sigma^{-1})$.
An argument in the proof of Proposition~\ref{proposition:computation-of-rho-invariant-of-J_0^i} shows that $\mu_{n-1}$ is nontrivial in $H_1(W_{n-1};\Q[t^{\pm 1}])$.
Note that $\mu_{n-1} \in H_1(W_{n-1};\Q[t^{\pm 1}])$ is $\lambda(t)$-torsion, since $\mu_{n-1}$ is identified with the axis of the satellite operator $P$ and $\lambda(t)$ is the Alexander polynomial of $P(U)$ that annihilates the Alexander module.
So, by our choice of $\Sigma$, the meridian $\mu_{n-1}$ is still nontrivial in $H_1(W_{n-1};\Q[t^{\pm 1}]\Sigma^{-1})$. Therefore $\mu_{n-1}$ is nontrivial in $H_1(\overline{W_{n-1}\sm Z};\Q[t^{\pm 1}]\Sigma^{-1})$.

We have $H_1(M((-Q(J))\# Q(U));\Q[t^{\pm 1}]\Sigma^{-1})=0$ by our choice of $\Sigma$ and by the assumption that the Alexander polynomial of $Q(U)$ is relatively prime to~$\lambda(t)$. Using a Mayer-Vietoris sequence for the pair $(\overline{W_{n-1}\sm Z},Z)$, one can readily see that 
\[
H_1(\overline{W_{n-1}\sm Z};\Q[t^{\pm 1}]\Sigma^{-1})\cong H_1(W_{n-1};\Q[t^{\pm 1}]\Sigma^{-1})/\Im H_1(Z;\Q[t^{\pm 1}]\Sigma^{-1}).
\]
Therefore, $\mu_{n-1}$ is nontrivial in $H_1(W_{n-1};\Q[t^{\pm 1}]\Sigma^{-1})/\Im H_1(Z;\Q[t^{\pm 1}]\Sigma^{-1})$. It follows that $\mu_{n-1}$ is nontrivial in $\cP^1\pi_1(W_{n-1})/\cP^2\pi_1(W_{n-1})$.

For $k\le n-2$, suppose $\mu_{k+1}^1$ is nontrivial in the quotient $\cP^{n-k-1}\pi_1(W_{k+1})/\cP^{n-k}\pi_1(W_{k+1})$. We use the following lemma, which will play the role of Lemma~\ref{lemma:Blanchfield-bordism} for the purpose of this section.

\begin{lemma}\label{lemma:coprime-Blanchfield-bordism}
Let $\pi=\pi_1(W_{k+1})$ and $\Gamma=\pi/\cP^{n-k}\pi$. Let $\psi\colon \pi\to \Gamma$ be the quotient map. Then, $(W_{k+1},\psi)$ is a $\Q$-coefficient Blanchfield bordism.
\end{lemma}
\begin{proof}
We have
\[
W_{k+1}=V\cup C \cup Z \cup E_{n-1}\cup E_{n-2}\cup \cdots \cup E_{k+1}.
\]

We have $\Gamma^{(n-k)}=\{e\}$ since $\pi^{(n-k)}\subset \cP^{n-k}\pi$. It follows from Lemma~\ref{lemma:integral-solution-lagrangian-bordism} that $(V,\psi)$ is a $\Q$-coefficient Lagrangian bordism since the slice disk exterior $V$ is an integral $(n-k)$-solution. Then, $(V\cup C,\psi)$ is a $\Q$-coefficient Lagrangian bordism by Lemmas~\ref{lemma:lagrangian-bordism-extension} and \ref{lemma:C-and-E-injectivity}. 
The integral $n$-solution $Z$ is also an integral $(n-k)$-solution, and therefore $(Z,\psi)$ is a $\Q$-coefficient Lagrangian bordism by Lemma~\ref{lemma:integral-solution-lagrangian-bordism}. We use the following lemma.
\begin{lemma}\label{lemma:solution-extension-Blanchfield-bordism}
Let $R=\Z_p$ or $\Q$. Let $W$ and $X$ be compact connected 4-manifolds with boundary such that $\partial X = W\cap X$ and it is a component of $\partial W$. Let $W'=W\cup X$ and let $\psi\colon \pi_1(W')\to G$ be a homomorphism where $G$ is a poly-torsion-free-abelian group. Suppose that $W'$ has nonempty boundary, and that the inclusion induces an injection $H_1(\partial X;R)\to H_1(X;R)$ and a trivial homomorphism $H_2(\partial X;R)\to H_2(X;R)$.
If each of $(W,\psi|_{\pi_1(W)})$ and $(X,\psi|_{\pi_1(X)})$ is an $R$-coefficient Lagrangian bordism, then so is $(W',\psi)$.
\end{lemma}
\begin{proof}
Suppose that elements $\ell_1$,~\ldots, $\ell_k$, $d_1$,~\ldots, $d_k$ in $H_2(W;RG)$ and elements $\ell_{k+1}$,~\ldots, $\ell_{k+m}$, $d_{k+1}$,~\ldots, $d_{k+m}$ in $H_2(X;RG)$ satisfy Definition~\ref{definition:lagrangian-bordism}~(1). Then, viewing the $\ell_i$ and $d_j$ as elements in $H_2(W';RG)$, $1\le i,j\le k+m$, they satisfy Definition~\ref{definition:lagrangian-bordism}~(1). 

Using the assumption on the maps $H_i(\partial X;R)\to H_i(X;R)$ for $i=1,2$ and a Mayer-Vietoris sequence of the pair $(W,X)$, one can readily see that 
\[
\rank_R \Coker \{H_2(\partial W';R)\rightarrow H_2(W';R)\} \le 2(k+m).
\]
Therefore, $(W',\psi)$ is an $R$-coefficient Lagrangian bordism.
\end{proof}
We continue to prove Lemma~\ref{lemma:coprime-Blanchfield-bordism}.
Note that $H_1(\partial Z;\Q)\to H_1(Z;\Q)$ is an isomorphism and $H_2(\partial Z;\Q)\to H_2(Z;\Q)$ is trivial since $Z$ is an integral $n$-solution.
By applying Lemma~\ref{lemma:solution-extension-Blanchfield-bordism} to $(W,X) = (V\cup C, Z)$, it follows that $(W_n,\psi)$ is a $\Q$-coefficient Lagrangian bordism.

Now, one can show that $(W_{k+1},\psi)$ is a $\Q$-coefficient Lagrangian bordism by repeatedly applying Lemmas~\ref{lemma:lagrangian-bordism-extension} and \ref{lemma:C-and-E-injectivity}. It follows from Proposition~\ref{proposition:lagrangian-Blanchfield-bordism} that $(W_{k+1},\psi)$ is a $\Q$-coefficient Blanchfield bordism.
\end{proof}
We continue to prove~(H6). Using the argument in the proof of Proposition~\ref{proposition:computation-of-rho-invariant-of-J_0^i} (with Lemma~\ref{lemma:coprime-Blanchfield-bordism} in place of Lemma~\ref{lemma:Blanchfield-bordism}), one can show that $\mu_k^1$ is nontrivial in the quotient $\cP^{n-k}\pi_1(W_k)/\cP^{n-k+1}\pi_1(W_k)$. By induction, the meridian $\mu=\mu_0$ of $J=J_0$ is nontrivial in $\cP^n\pi_1(W_0)/\cP^{n+1}\pi_1(W_0)$. Therefore, $\phi$ satisfies~(H6).
 \def\qedsymbol{}
\end{proof}
Therefore, $\phi$ satisfies both (H5) and~(H6).
It completes the proof of Proposition~\ref{proposition:phi-coprime}.
\end{proof}

\subsubsection*{Generalization of Theorem~\ref{theorem:coprime-satellites}}
Theorem~\ref{theorem:coprime-satellites} generalizes to the case of multi-composition of distinct satellite operators, as stated below. 

\begin{theorem}[{Generalization of Theorem~\ref{theorem:coprime-satellites}}]\label{theorem:coprime-satellites-generalization}
For $1\le k\le n$, let $P_k\colon \cC\to \cC$ be a winding number zero satellite operator with axis $\eta_k$ such that $\Bl_k(\eta_k,\eta_k)$ is nontrivial where $\Bl_k$ is the Blanchfield form for $P_k(U)$. If $Q\colon \cC\to \cC$ is a winding number zero satellite operator such that the Alexander polynomials of $P_n(U)$ and $Q(U)$ are relatively prime in $\Q[t^{\pm 1}]$, then the operators $P_n\circ P_{n-1}\circ \cdots \circ P_1$ and $Q$ are distinct. 
\end{theorem}
\begin{proof}
One can prove Theorem~\ref{theorem:coprime-satellites-generalization} following the proof of Theorem~\ref{theorem:coprime-satellites}, with changes listed below. For brevity, let $\tilde{P}_k = P_k\circ P_{k-1}\circ \cdots \circ P_1$. 
\begin{enumerate}
	\item Let $L_1$ and $L_2$ be constants such that  $|\rhot(M(P_k(U)),\phi)| < L_1$ for all $\phi$ and $1\le k\le n$, and $|\rho(M(-\tilde{P}_n(U), \phi)| < L_2$ for all $\phi$.
	\item Replace the operator $P^k$ with $\tilde{P}_k$, $1\le k\le n$. 
	\item For $0\le k\le n-1$, let $E_k = E(P_{k+1}, \tilde{P}_k(J))$. Then,
	\[
	\partial E_k=M(\tilde{P}_k(J))\sqcup M(P_{k+1}(U)) \sqcup (-M(\tilde{P}_{k+1}(J))).
	\] 
	\item In the proof of Proposition~\ref{proposition:phi-coprime}, let $\lambda(t)$ be the Alexander polynomial of $P_n(U)$.
	\qedhere
\end{enumerate}

\end{proof}


\appendix

\section{Slice satellite operators}\label{section:slice-satellite-operators}

This section is independent of the rest of the paper.  The aim of this section is to show how one can readily prove some special cases of Theorem~\ref{theorem:main} using the proof of \cite[Theorem~4.11]{Cha:2010-01}.
Our work is partially motivated by this.

For example, in Theorem~\ref{theorem:main-special-case} below we show that a simple modification of the proof of \cite[Theorem~4.11]{Cha:2010-01} gives a proof of Theorem~\ref{theorem:main} with an additional assumption that $P(U)$ is slice.
In Theorem~\ref{theorem:infinite-rank}, we prove that a (non-iterated) satellite operator $P\colon \cC \to \cC$ has infinite rank image by combining Theorem~\ref{theorem:main-special-case} and elementary properties of abelian groups under the assumption for Theorem~\ref{theorem:main}, without assuming that $P(U)$ is slice.

\subsection{A special case of Theorem~\ref{theorem:main} for slice satellite operators}
\label{subsection:slice-satellite-operators}

In the proof of \cite[Theorem~4.11]{Cha:2010-01} the following was shown implicitly. Let $P\colon \cC\to \cC$ be a winding number zero satellite operator with axis $\eta$. Suppose that $P(U)$ is slice and $H_1(M(P(U));\Z[t^{\pm 1}])$ is nontrivial and generated by the axis $\eta$. Then, for given $n\ge 1$ there exist knots $J_0^i$, $i\ge 1$, such that the set $\{P^n(J_0^i)\}_{i\ge 1}$ is linearly independent in $\cF_n/\cF_{n.5}$. Note that since the Blanchfield form $\Bl$ of $P(U)$ is nonsingular, if the $\eta$ generates $H_1(M(P(U));\Z[t^{\pm 1}])$, then $\Bl(\eta, \eta)$ is nontrivial. Now we show that a simple modification of the proof of \cite[Theorem~4.11]{Cha:2010-01}  gives the following theorem,

\begin{theorem}[{Generalization of \cite[Theorem~4.11]{Cha:2010-01}}] \label{theorem:Cha-slice-infinite-rank}
Let $n\ge 1$. Let $P\colon \cC\to \cC$ be a winding number zero satellite operator with axis $\eta$. Suppose that $P(U)$ is slice and $\Bl(\eta,\eta)$ is nontrivial. Then, there exist infinitely many knots $J_0^i$, $i=1,2,\ldots$, such that the set $\{P^n(J_0^i)\}_{i\ge 1}$ is linearly independent in $\cF_n/\cF_{n.5}$.
\end{theorem}

\begin{proof}
We list observations and modifications on the proof of \cite[Theorem~4.11]{Cha:2010-01} which are needed to obtain Theorem~\ref{theorem:Cha-slice-infinite-rank}.

\begin{enumerate}
	\item We use the same $J_0^i$ as in \cite[Theorem~4.11]{Cha:2010-01}, which were obtained in \cite[Proposition~4.12]{Cha:2010-01}. (For our purpose, the knots $J_0^i$ need not satisfy (C3) in \cite[Proposition~4.12]{Cha:2010-01}.)
	\item In \cite[Theorem~4.11]{Cha:2010-01} we take $K_k=P(U)$ and $\eta_k=\eta$ for all $k$. Then $J_n^i$ therein is $P^n(J_0^i)$ for each $i$. 
	\item For a prime $p$, let $\Bl^p$ denote the Blanchfield form  
	\[
	\Bl\nolimits^p\colon H_1(M(P(U));\Z_p[t^{\pm 1}])\times H_1(M(P(U));\Z_p[t^{\pm 1}])\to \Z_p(t)/\Z_p[t^{\pm 1}].
	\] 
	In \cite[Theorem~1.4]{Cha:2010-01} it was assumed that $P(U)$ is slice and the $\eta$ generates the Alexander module $H_1(M(P(U));\Z[t^{\pm 1}])$; this assumption was used in the first paragraph in \cite[p.~4801]{Cha:2010-01}, but the only  property of $\eta$ one actually needs is that $\Bl(\eta, \eta)$ is nontrivial and $\Bl^{p_i}(\eta,\eta)$ is nontrivial for each prime $p_i$ associated with $J_0^i$. 
	\item If $\Bl(\eta,\eta)=f(t)/g(t)\in \Q(t)/\Z[t^{\pm 1}]$ for some $f(t),g(t)\in \Z[t^{\pm 1}]$, $\deg f < \deg g$, then one can readily see that $\Bl^p(\eta,\eta)$ is nontrivial for any prime $p$ that is greater than the absolute values of all the coefficients of $f(t)$ and~$g(t)$.
	Therefore, if $\Bl(\eta, \eta)$ is nontrivial, then there exist infinitely many primes $p_1 <p_2 < \cdots$ such that the primes $p_i$ satisfy (C1)--(C3) in \cite[Proposition~4.12]{Cha:2010-01} and $\Bl^{p_i}(\eta, \eta)$ is nontrivial for each $p_i$. We use these primes $p_i$ instead of the primes $p_i$ in the proof of \cite[Theorem~4.11]{Cha:2010-01}.
\qedhere
\end{enumerate}
\end{proof}

Using Theorem~\ref{theorem:Cha-slice-infinite-rank} one can obtain the following theorem.
\begin{theorem}[{Special case of Theorem~\ref{theorem:main} when $P(U)$ is slice}]\label{theorem:main-special-case}
Let $P\colon \cC \to \cC$ be a winding number zero satellite operator with axis $\eta$. If $P(U)$ is slice and $\Bl(\eta,\eta)$ is nontrivial, then $\langle P^n(\cC)\rangle/\langle P^{n+1}(\cC)\rangle$ has infinite rank for each $n\ge 0$. 
\end{theorem}
\begin{proof}
The case $n=0$ was proven in Proposition~\ref{proposition:nonzero-winding-number}. For $n\ge 1$, let $J_0^i$, $i\ge 1$, be the knots in Theorem~\ref{theorem:Cha-slice-infinite-rank}. Then obviously $P^n(J_0^i)\in P^n(\cC)$ for all $i$ and they are linearly independent in $\cC/\cF_{n.5}$. Since $P(U)$ is slice and $P$ has winding number zero, by Lemma~\ref{lemma:construction-of-n-solvable-knots} it follows that $\langle P^{n+1}(\cC)\rangle\subset \cF_{n+1}\subset \cF_{n.5}$. Therefore, $P^n(J_0^i)$ are linearly independent in $\langle P^n(\cC)\rangle/\langle P^{n+1}(\cC)\rangle$.
\end{proof}

\subsection{Satellites of infinite rank}\label{subsection:affine-linear-independence}

The following statement is a consequence of Theorem~\ref{theorem:main}:

\begin{theorem}\label{theorem:infinite-rank}
Let $P\colon \cC\to \cC$ be a winding number zero satellite operator with axis $\eta$. If $\Bl(\eta,\eta)$ is nontrivial, then $\langle P(\cC)\rangle$ has infinite rank.
\end{theorem}

In what follows, we observe that Theorem~\ref{theorem:infinite-rank} can also be obtained by using Theorem~\ref{theorem:main-special-case} and elementary properties of abelian groups.

\begin{proof}[An alternative proof of Theorem~\ref{theorem:infinite-rank}]
We use the following lemma.

\begin{lemma}\label{lemma:affine-linear-independence}
Let $G$ be an abelian group and $\{a_i\}_{i\ge 0}$ be an infinite sequence of elements in~$G$. If $\{a_i-a_0\mid i\ge 1\}$ is linearly independent in $G$, then there exists $N\ge 0$ such that $\{a_i\mid i> N\}$ is linearly independent in $G$.
\end{lemma}
\begin{proof}
There is an integer $N>0$ such that
\begin{equation}
  \label{equation:choice-of-N}
  \langle a_0 \rangle\cap \langle a_1-a_0,a_2-a_0,\ldots \rangle\subset \langle a_1-a_0,a_2-a_0,\ldots, a_N-a_0\rangle  
\end{equation}
since the left-hand side has rank at most one.
We assert that $\{a_i\mid i> N\}$ is linearly independent in $G$.
Suppose that 
\[
c_1a_{i_1}+c_2a_{i_2}+\cdots +c_na_{i_n}=0\]
for some $c_j\in \Z$ and $i_j$ such that $N<i_1<i_2<\cdots <i_n$.
Then
\[
c_1(a_{i_1}-a_0) + c_2(a_{i_2}-a_0)+ \cdots +c_n(a_{i_n}-a_0)=-\Biggl(\sum_{j=1}^nc_j\Biggr) a_0.
\]
By our choice of $N$, the left-hand side linear combination lies in the right-hand side of~\eqref{equation:choice-of-N}.
Since $\{a_i-a_0\mid i\ge 1\}$ is linearly independent, it follows that every $c_i$ vanishes.
\end{proof}

Now let $J_0^i$ be the knots given in Theorem~\ref{theorem:Cha-slice-infinite-rank}. Let $K_i=P(J_0^i)$ for $i=1,2,\ldots$ and let $K_0=P(U)$. 
Let $Q:=P\# (-P(U))\colon \cC \to \cC$ be the satellite operator defined by 
\[
Q(J)=P(J)\#(-P(U)).
\]
Then, $K_i\#(-K_0)=Q(J_0^i)$ in $\cC$. Note that $Q(U)=P(U)\#(-P(U))$ is slice and $\Bl(\eta,\eta)$ is still nontrivial for the Blanchfield form $\Bl$ of $Q(U)$. From Theorem~\ref{theorem:main-special-case} and its proof, it follows that $Q(J_0^i)=K_i\#(-K_0)$, $i\ge 1$, are linearly independent in $\langle Q(\cC)\rangle/\langle Q^2(\cC)\rangle$, and hence in~$\cC$.
By Lemma~\ref{lemma:affine-linear-independence}, there exists $N>0$ such that $\{K_i\mid i> N\}$ is linearly independent in~$\cC$.
Therefore $\langle\cP(\cC)\rangle$ has infinite rank in~$\cC$.
\end{proof}

Theorem~\ref{theorem:infinite-rank} has an interesting application.
As discussed at the end of the introduction, we prove the following generalization of the main results of~\cite{Livingston:2002-1,Kim:2005-1}.


\begin{corollary}\label{corollary:infinite-rank-fixed-Seifert-form}
Suppose that $V$ is a Seifert form of a knot $K$ with nontrivial Alexander polynomial. Then, there exist infinitely many knots $K_i$ with the same Seifert form $V$ which are linearly independent in~$\cC$. 
\end{corollary}

\begin{proof}
Let $B\ell$ be the Blanchfield form for~$K$.
By Lemma~\ref{lemma:nontrivial-self-blanchfield} below, there exists $\eta\in H_1(M(K);\Z[t^{\pm 1}])$ such that $B\ell (\eta,\eta)$ is nontrivial.
Choose an embedded circle $\eta\in S^3\sm K$ that is unknotted in $S^3$ and represents the element $\eta \in H_1(M(K);\Z[t^{\pm 1}])$.
Identify the exterior of $\eta\subset S^3$ with $S^1\times D^2$ along the zero framing, to view $P := K \subset S^1\times D^2$ as a satellite operator.
By Theorem~\ref{theorem:infinite-rank} and its proof, there exist infinitely many knots $J_0^i$, $i>N$, such that $K_i=P(J_0^i)$ are linearly independent in~$\cC$.
The knots $K_i$ have the same Seifert form as $K=P(U)$ since $P$ has winding number zero.
\end{proof}

\begin{lemma}
  \label{lemma:nontrivial-self-blanchfield}
  Suppose that $K$ is a knot in $S^3$ with nontrivial Alexander polynomial.
  Let $B\ell$ be the Blanchfield form for~$K$.
  Then there is $\eta \in H_1(M(K);\Z[t^{\pm1}])$ such that $\Bl(\eta,\eta)\ne 0$.
\end{lemma}

\begin{proof}
  Let $\Delta_K(t) \in \Z[t]$ be the Alexander polynomial of~$K$.
  Let $B$ be the rational Blanchfield form on the rational Alexander module $H_1(M(K);\Q[t^{\pm1}]) = H_1(M(K);\Z[t^{\pm1}]) \otimes \Q$.
  The value of $B$ lies in~$\Q(t)/\Q[t^{\pm1}]$.
  It suffices to show that $B(\eta,\eta)\ne 0$ for some $\eta \in H_1(M(K);\Q[t^{\pm1}])$.
  Recall that $\Q[t^{\pm1}]$ is a PID and $B$ is nonsingular.
  Since $\Delta_K(t)$ is nontrivial, it follows that there are $x,y\in H_1(M(K);\Q[t^{\pm 1}])$ such that $B(x,y)=1/f(t) \bmod \Q[t^{\pm1}]$ for some factor $f(t)\in \Z[t]$ of~$\Delta_K(t)$ with $\deg f \ge 1$, $f(0)\ne 0$.
  If $B(x,x)\ne 0$ or $B(y,y)\ne 0$, the conclusion holds.
  Suppose that $B(x,x)=B(y,y)=0$.
  Then we have
  \[
    B(x+y,x+y)=\frac1{f(t)} + \frac1{f(t^{-1})}  = \frac{f(t)+f(t^{-1})}{f(t)f(t^{-1})} \mod \Q[t^{\pm1}].  
  \]
  Suppose $B(x+y,x+y)=0$.
  Then $f(t)f(t^{-1}) \mid f(t)+f(t^{-1})$, and thus $f(t) \mid f(t^{-1}) \mid f(t)$ in $\Q[t^{\pm1}]$.
  It follows that $f(t^{-1}) = a t^{-d}f(t)$ where $d=\deg f$ and $a\in \Q$ is nonzero.
  So $B(x+y,x+y)=(a+t^d)/af(t)=0$, and consequently $f(t) = k(a+t^d)$ for some nonzero~$k$.
  Since $k(a+t^{-d}) = f(t^{-1}) = a t^{-d}f(t) = k(a + a^2 t^{-d})$, it follows that $a=\pm1$.
  Since $f(t)\in\Z[t]$, $k$ is an integer.
  We have $f(1)=0$ or $2k$.
  This is a contradiction since $f(t) \mid \Delta_K(t)$ in $\Z[t]$ and $\Delta_K(1)=\pm1$. 
  Thus $B(x+y,x+y)\ne 0$ and the conclusion holds.    
\end{proof}

\subsubsection*{Acknowledgment} The authors thank the anonymous referee for many helpful comments and inspiring questions.

\bibliographystyle{amsalpha}
\def\MR#1{}
\bibliography{research}

\newcommand{\etalchar}[1]{$^{#1}$}
\providecommand{\bysame}{\leavevmode\hbox to3em{\hrulefill}\thinspace}
\providecommand{\MR}{\relax\ifhmode\unskip\space\fi MR }
\providecommand{\MRhref}[2]{%
  \href{http://www.ams.org/mathscinet-getitem?mr=#1}{#2}
}
\providecommand{\href}[2]{#2}
\begin{thebibliography}{DHMS22}

\bibitem[Bla57]{Blanchfield:1957-1}
R.~C. Blanchfield, \emph{Intersection theory of manifolds with operators with applications to knot theory}, Ann. of Math. (2) \textbf{65} (1957), 340--356. \MR{19,53a}

\bibitem[CDR14]{Cochran-Davis-Ray:2014-1}
Tim~D. Cochran, Christopher~W. Davis, and Arunima Ray, \emph{Injectivity of satellite operators in knot concordance}, J. Topol. \textbf{7} (2014), no.~4, 948--964. \MR{3286894}

\bibitem[CG85]{Cheeger-Gromov:1985-1}
J.~Cheeger and M.~Gromov, \emph{Bounds on the von {N}eumann dimension of {$L\sp 2$}-cohomology and the {G}auss-{B}onnet theorem for open manifolds}, J. Differential Geom. \textbf{21} (1985), no.~1, 1--34. \MR{MR806699 (87d:58136)}

\bibitem[Cha09]{Cha:2009-1}
Jae~Choon Cha, \emph{Structure of the string link concordance group and {H}irzebruch-type invariants}, Indiana Univ. Math. J. \textbf{58} (2009), no.~2, 891--927. \MR{MR2514393}

\bibitem[Cha13]{Cha:2010-01}
\bysame, \emph{Amenable ${L}^2$-theoretic methods and knot concordance}, Int. Math. Res. Not. IMRN (2013), no.~15, 1--36.

\bibitem[Cha14]{Cha:2014-1}
\bysame, \emph{Symmetric {W}hitney tower cobordism for bordered 3-manifolds and links}, Trans. Amer. Math. Soc. \textbf{366} (2014), no.~6, 3241--3273. \MR{3180746}

\bibitem[Cha16]{Cha:2014-2}
\bysame, \emph{A topological approach to {C}heeger-{G}romov universal bounds for von {N}eumann {$\rho$}-invariants}, Comm. Pure Appl. Math. \textbf{69} (2016), no.~6, 1154--1209. \MR{3493628}

\bibitem[Cha21]{Cha:2021-1}
\bysame, \emph{Primary decomposition in the smooth concordance group of topologically slice knots}, Forum Math. Sigma \textbf{9} (2021), Paper No. e57, 37. \MR{4299596}

\bibitem[Che23]{Chen:2019-2}
Wenzhao Chen, \emph{Knot {F}loer homology of satellite knots with {$(1,1)$} patterns}, Selecta Math. (N.S.) \textbf{29} (2023), no.~4, Paper No. 53, 63. \MR{4615610}

\bibitem[CHH13]{Cochran-Harvey-Horn:2012-1}
Tim~D. Cochran, Shelly Harvey, and Peter Horn, \emph{Filtering smooth concordance classes of topologically slice knots}, Geom. Topol. \textbf{17} (2013), no.~4, 2103--2162. \MR{3109864}

\bibitem[CHL09]{Cochran-Harvey-Leidy:2009-01}
T.~D. Cochran, S.~Harvey, and C.~Leidy, \emph{Knot concordance and higher-order {B}lanchfield duality}, Geom. Topol. \textbf{13} (2009), no.~3, 1419--1482. \MR{MR2496049}

\bibitem[CHL11]{Cochran-Harvey-Leidy:2009-02}
\bysame, \emph{Primary decomposition and the fractal nature of knot concordance}, Math. Ann. \textbf{351} (2011), no.~2, 443--508.

\bibitem[CK02]{Cha-Ko:2000-1}
Jae~Choon Cha and Ki~Hyoung Ko, \emph{Signatures of links in rational homology homology spheres}, Topology \textbf{41} (2002), 1161--1182.

\bibitem[CK21]{Cha-Kim:2021-1}
Jae~Choon Cha and Min~Hoon Kim, \emph{The bipolar filtration of topologically slice knots}, Adv. Math. \textbf{388} (2021), Paper No. 107868, 32. \MR{4283760}

\bibitem[CK23]{Cahn-Kjuchukova:2023-1}
Patricia Cahn and Alexandra Kjuchukova, \emph{Linking in cyclic branched covers and satellite (non)-homomorphisms}, arXiv preprint arXiv:2308.05856 (2023).

\bibitem[CL04]{Cha-Livingston:2002-1}
Jae~Choon Cha and Charles Livingston, \emph{Knot signature functions are independent}, Proc. Amer. Math. Soc. \textbf{132} (2004), no.~9, 2809--2816 (electronic). \MR{MR2054808 (2005d:57004)}

\bibitem[CO12]{Cha-Orr:2009-01}
Jae~Choon Cha and Kent~E. Orr, \emph{${L}^2$-signatures, homology localization, and amenable groups}, Comm. Pure Appl. Math. \textbf{65} (2012), 790--832.

\bibitem[COT03]{Cochran-Orr-Teichner:2003-1}
Tim~D. Cochran, Kent~E. Orr, and Peter Teichner, \emph{Knot concordance, {W}hitney towers and {$L\sp 2$}-signatures}, Ann. of Math. (2) \textbf{157} (2003), no.~2, 433--519. \MR{1 973 052}

\bibitem[COT04]{Cochran-Orr-Teichner:2004-1}
\bysame, \emph{Structure in the classical knot concordance group}, Comment. Math. Helv. \textbf{79} (2004), no.~1, 105--123. \MR{MR2031301 (2004k:57005)}

\bibitem[CT07]{Cochran-Teichner:2003-1}
Tim~D. Cochran and Peter Teichner, \emph{Knot concordance and von {N}eumann {$\rho$}-invariants}, Duke Math. J. \textbf{137} (2007), no.~2, 337--379. \MR{MR2309149 (2008f:57005)}

\bibitem[CW03]{Chang-Weinberger:2003-1}
S.~Chang and S.~Weinberger, \emph{On invariants of {H}irzebruch and {C}heeger-{G}romov}, Geom. Topol. \textbf{7} (2003), 311--319 (electronic). \MR{MR1988288 (2004c:57052)}

\bibitem[DHMS22]{Dai-Hedden-Mallick-Stoffregen:2022-1}
Irving Dai, Matthew Hedden, Abhishek Mallick, and Matthew Stoffregen, \emph{Rank-expanding satellites, {W}hitehead doubles, and {H}eegaard {F}loer homology}, arXiv preprint arXiv:2209.07512, to appear in J. Topol. (2022).

\bibitem[DIS{\etalchar{+}}22]{Daemi-Imori-Sato-Scaduto-Taniguchi:2022-1}
Aliakbar Daemi, Hayato Imori, Kouki Sato, Christopher Scaduto, and Masaki Taniguchi, \emph{Instantons, special cycles, and knot concordance}, arXiv preprint arXiv:2209.05400, to appear in Geom. Topol. (2022).

\bibitem[Don83]{Donaldson:1983-1}
S.~K. Donaldson, \emph{An application of gauge theory to four-dimensional topology}, J. Differential Geom. \textbf{18} (1983), no.~2, 279--315. \MR{710056}

\bibitem[Fra13]{Franklin:2013-1}
Bridget~D. Franklin, \emph{The effect of infecting curves on knot concordance}, Int. Math. Res. Not. IMRN (2013), no.~1, 184--217. \MR{3041699}

\bibitem[Fre84]{Freedman:1984-1}
M.~H. Freedman, \emph{The disk theorem for four-dimensional manifolds}, Proceedings of the {I}nternational {C}ongress of {M}athematicians, {V}ol.\ 1, 2 ({W}arsaw, 1983) (Warsaw), PWN, 1984, pp.~647--663. \MR{MR804721 (86m:57016)}

\bibitem[Fri03]{Friedl:2003-7}
Stefan~Klaus Friedl, \emph{Eta invariants as sliceness obstructions and their relation to {C}asson-{G}ordon invariants}, ProQuest LLC, Ann Arbor, MI, 2003, Thesis (Ph.D.)--Brandeis University. \MR{2704824}

\bibitem[Gil82]{Gilmer:1982-1}
P.~M. Gilmer, \emph{On the slice genus of knots}, Invent. Math. \textbf{66} (1982), no.~2, 191--197. \MR{MR656619 (83g:57003)}

\bibitem[Gil83]{Gilmer:1983-1}
\bysame, \emph{Slice knots in ${S}\sp{3}$}, Quart. J. Math. Oxford Ser. (2) \textbf{34} (1983), no.~135, 305--322. \MR{85d:57004}

\bibitem[GL92]{Gilmer-Livingston:1992-1}
P.~M. Gilmer and C.~Livingston, \emph{Discriminants of {C}asson-{G}ordon invariants}, Math. Proc. Cambridge Philos. Soc. \textbf{112} (1992), no.~1, 127--139. \MR{94e:57007}

\bibitem[Har08]{Harvey:2006-1}
S.~Harvey, \emph{Homology cobordism invariants and the {C}ochran-{O}rr-{T}eichner filtration of the link concordance group}, Geom. Topol. \textbf{12} (2008), 387--430.

\bibitem[Hed07]{Hedden:2007-1}
Matthew Hedden, \emph{Knot {F}loer homology of {W}hitehead doubles}, Geom. Topol. \textbf{11} (2007), no.~4, 2277--2338. \MR{2372849 (2008m:57030)}

\bibitem[HK12]{Hedden-Kirk:2010-1}
Matthew Hedden and Paul Kirk, \emph{Instantons, concordance, and {W}hitehead doubling}, J. Differential Geom. \textbf{91} (2012), no.~2, 281--319. \MR{2971290}

\bibitem[HKL16]{Hedden-Kim-Livingston:2016-1}
M.~Hedden, S.~G. Kim, and C.~Livingston, \emph{Topologically slice knots of smooth concordance order two}, J. Differential Geom. \textbf{102} (2016), no.~3, 353--393. \MR{3466802}

\bibitem[HLR12]{Hedden-Livingston-Ruberman:2010-01}
Matthew Hedden, Charles Livingston, and Daniel Ruberman, \emph{Topologically slice knots with nontrivial {A}lexander polynomial}, Adv. Math. \textbf{231} (2012), no.~2, 913--939. \MR{2955197}

\bibitem[HPC21]{Hedden-Pinzon-Caicedo:2021-1}
Matthew Hedden and Juanita Pinz\'{o}n-Caicedo, \emph{Satellites of infinite rank in the smooth concordance group}, Invent. Math. \textbf{225} (2021), no.~1, 131--157. \MR{4270665}

\bibitem[JKM23]{Johanningsmeier-Kim-Miller:2023-1}
Randall Johanningsmeier, Hillary Kim, and Allison~N Miller, \emph{A partial resolution of hedden's conjecture on satellite homomorphisms}, arXiv preprint arXiv:2308.06890 (2023).

\bibitem[Kim05]{Kim:2005-1}
Taehee Kim, \emph{An infinite family of non-concordant knots having the same {S}eifert form}, Comment. Math. Helv. \textbf{80} (2005), no.~1, 147--155. \MR{2130571}

\bibitem[Lei06]{Leidy:2006-1}
C.~Leidy, \emph{Higher-order linking forms for knots}, Comment. Math. Helv. \textbf{81} (2006), no.~4, 755--781. \MR{MR2271220 (2007g:57011)}

\bibitem[Lit79]{Litherland:1979-1}
R.~A. Litherland, \emph{Signatures of iterated torus knots}, Topology of low-dimensional manifolds (Proc. Second Sussex Conf., Chelwood Gate, 1977), Springer, Berlin, 1979, pp.~71--84. \MR{80k:57012}

\bibitem[Lit84]{Litherland:1984-1}
R.~A. Litherland, \emph{Cobordism of satellite knots}, Four-manifold theory ({D}urham, {N}.{H}., 1982), Contemp. Math., vol.~35, Amer. Math. Soc., Providence, RI, 1984, pp.~327--362. \MR{780587}

\bibitem[Liv83]{livingston:1983-1}
Charles Livingston, \emph{Knots which are not concordant to their reverses}, Quart. J. Math. Oxford Ser. (2) \textbf{34} (1983), no.~135, 323--328. \MR{711524}

\bibitem[Liv02]{Livingston:2002-1}
C.~Livingston, \emph{Seifert forms and concordance}, Geom. Topol. \textbf{6} (2002), 403--408 (electronic). \MR{MR1928840 (2003f:57019)}

\bibitem[Liv22]{Livingston:2022-1}
Charles Livingston, \emph{Rank-expanding satellite operators on the topological knot concordance group}, arXiv preprint arXiv:2211.04272, to appear in Michigan Math. J. (2022).

\bibitem[LMPC22]{Lidman-Miller-Pinzon-Caicedo:2022-1}
Tye Lidman, Allison~N Miller, and Juanita Pinz{\'o}n-Caicedo, \emph{Linking number obstructions to satellite homomorphisms}, arXiv preprint arXiv:2207.14198, to appear in Quantum Topol. (2022).

\bibitem[LS03]{Lueck-Schick:2003-1}
W.~L{\"u}ck and T.~Schick, \emph{Various {$L\sp 2$}-signatures and a topological {$L\sp 2$}-signature theorem}, High-dimensional manifold topology, World Sci. Publishing, River Edge, NJ, 2003, pp.~362--399. \MR{MR2048728 (2005b:58027)}

\bibitem[Mil23]{Miller:2019-01}
Allison Miller, \emph{Homomorphism obstructions for satellite maps}, Transactions of the American Mathematical Society, Series B \textbf{10} (2023), no.~08, 220--240.

\bibitem[Str74]{Strebel:1974-1}
R.~Strebel, \emph{Homological methods applied to the derived series of groups}, Comment. Math. Helv. \textbf{49} (1974), 302--332. \MR{MR0354896 (50 \#7373)}

\bibitem[Wan16]{Wang:2016-1}
Shida Wang, \emph{The genus filtration in the smooth concordance group}, Pacific J. Math. \textbf{285} (2016), no.~2, 501--510. \MR{3575577}

\end{thebibliography}

\end{document}